\documentclass[english,11pt,reqno]{article}

\usepackage{tikz}
\usetikzlibrary{shapes}
\usetikzlibrary{decorations.pathreplacing}
\usepackage{enumerate}
\usepackage[shortlabels]{enumitem}
\usepackage{verbatim}
\usepackage[margin=1in]{geometry}
\usepackage{amssymb}
\usepackage{caption}
\usepackage[normalem]{ulem}
\usepackage{bbm}
\usepackage{amsthm}
\usepackage{enumitem}
\usepackage{scalerel}    
\usepackage{stmaryrd}   

\usetikzlibrary{arrows.meta}

\DeclareFontFamily{U}{mathx}{\hyphenchar\font45}
\DeclareFontShape{U}{mathx}{m}{n}{
      <5> <6> <7> <8> <9> <10>
      <10.95> <12> <14.4> <17.28> <20.74> <24.88>
      mathx10
      }{}
\DeclareSymbolFont{mathx}{U}{mathx}{m}{n}
\DeclareFontSubstitution{U}{mathx}{m}{n}
\DeclareMathAccent{\widecheck}{0}{mathx}{"71}
\DeclareMathAccent{\wideparen}{0}{mathx}{"75}

\usepackage{rotating,centernot,cancel}    

\newcommand{\myfootnote}[1]{
    \renewcommand{\thefootnote}{}
    \footnotetext{\scriptsize#1}
    \renewcommand{\thefootnote}{\arabic{footnote}}
}

\usepackage{amsmath}
\usepackage{tocloft}
\usepackage{float}

\usepackage{tcolorbox}   
\PassOptionsToPackage{hyphens}{url}\usepackage{hyperref}
\usepackage{babel}
\theoremstyle{plain}

\input amssym.def
\input amssym.tex

\def\beqn{\begin{eqnarray}}
\def\eeqn{\end{eqnarray}}

\def\beq{\begin{equation}}
\def\eeq{\end{equation}}
\usepackage{mathtools}
\DeclarePairedDelimiter\floor{\lfloor}{\rfloor}

\newtheorem{theorem}{Theorem}[section]
\newtheorem{corollary}[theorem]{Corollary} 
\newtheorem{lemma}[theorem]{Lemma} 
\newtheorem*{lemma*}{Lemma}
\newtheorem{proposition}[theorem]{Proposition} 

\theoremstyle{remark}
\newtheorem{remark}[theorem]{Remark}

\theoremstyle{definition}

\numberwithin{figure}{section}
\numberwithin{equation}{section}
\def\exittime{{\bf Z}}

\def\R{\mathbb R}

\def\P{\mathbb P}  
\def\Z{\mathbb Z}

\def\to{\rightarrow}
\def\dlim[#1][#2]{\lim_{#1 \to #2, #1 \neq #2}}

\def\Var{\text{$\mathbb{V}$ar}}
\def\Exp{\text{Exp}}

\newcommand{\be}{\begin{equation}}
\newcommand{\ee}{\end{equation}}
\newcommand{\lzb}{\llbracket}   
\newcommand{\rzb}{\rrbracket}   

\newcommand{\nn}{\nonumber}
\providecommand{\abs}[1]{\vert#1\vert}

\newcommand{\fl}[1]{\lfloor{#1}\rfloor}

\def\ind{\mathbf{1}}

\def\w{\omega} 

\def\wt{\widetilde}    \def\wc{\widecheck}

\usepackage{xargs}
\usepackage[colorinlistoftodos,prependcaption,textsize=footnotesize]{todonotes}
\newcommandx{\addmath}[2][1=]{\todo[linecolor=red,backgroundcolor=red!25,bordercolor=red,#1]{#2}}
\newcommandx{\fixtext}[2][1=]{\todo[linecolor=blue,backgroundcolor=blue!25,bordercolor=blue,#1]{#2}}
\newcommandx{\note}[2][1=]{\todo[linecolor=yellow,backgroundcolor=yellow!25,bordercolor=yellow,#1]{#2}}

  \def\cD{\mathcal{D}}

\def\cL{\mathcal{L}}

\def\bgeod#1#2{\mathbf{b}^{#1, #2}}

\newcommand\bbullet{{{\scaleobj{0.6}{\bullet}}}} 

\newcommand\dbullet{{\scaleobj{0.7}{\bullet}}}   

\def\coal{{\mathbf z}} 

\def\evec{e}    

\begin{document}

\title{Coalescence estimates for the corner growth model with exponential weights}
\author{Timo Sepp\"{a}l\"{a}inen\\[2pt]{\small University of Wisconsin--Madison}\\{\small seppalai@math.wisc.edu}\\{\small\tt http://www.math.wisc.edu/$\sim$seppalai}  \and Xiao Shen\\[2pt]{\small University of Wisconsin--Madison}\\{\small xshen@math.wisc.edu}\\{\small\tt http://www.math.wisc.edu/$\sim$xshen}}
\date{}


\maketitle

\myfootnote{Date: \today}
\myfootnote{2010 Mathematics Subject Classification. 60K35, 	60K37}
\myfootnote{Key words and phrases: coalescence, exit time, fluctuation exponent, geodesic, last-passage percolation, Kardar-Parisi-Zhang, random growth model.}
\myfootnote{Address: Mathematics Department, University of Wisconsin-Madison, Van Vleck Hall, 480 Lincoln Dr., Madison WI 53706-1388, USA.}
\myfootnote{T.\ Sepp\"al\"ainen was partially supported by  National Science Foundation grants  DMS-1602486  and DMS-1854619, and by the Wisconsin Alumni Research Foundation.}

\begin{abstract}   We establish estimates for the coalescence time of semi-infinite directed geodesics in the planar corner growth model with i.i.d.\ exponential weights. There are four estimates: upper and lower bounds on the probabilities of both fast and slow coalescence on the correct spatial scale with exponent $3/2$. Our proofs utilize a geodesic duality introduced by Pimentel and properties of the increment-stationary last-passage percolation process. For fast coalescence our bounds are new and they have matching optimal exponential order of magnitude. For slow coalescence we reproduce bounds proved earlier with integrable probability inputs, except that our upper bound misses the optimal order by a logarithmic factor.
\end{abstract}

\tableofcontents
\setcounter{tocdepth}{0}

\section{Introduction}

Random growth models of the first- and last-passage type have been a central part of the mathematical theory  of  spatial stochastic processes since the seminal work of Eden \cite{eden1961} and Hammersley and Welsh \cite{Ham-Wel-65}.   In these models,  growth proceeds along optimal paths called {\it geodesics}, determined by a random environment.  
The interesting and challenging  objects of study  are the {\it directed semi-infinite geodesics}.  These  pose an immediate existence question because  they are asymptotic objects and hence cannot be defined locally in a simple manner.   Once the existence question is resolved, questions concerning their multiplicity and geometric behavior such as coalescence arise.  

Techniques for establishing  the existence, uniqueness, and coalescence of semi-infinite geodesics were first introduced by Newman and co-authors in the 1990s \cite{How-New-97, How-New-01,fppcoal2, buse2} in the context of planar undirected first-passage percolation (FPP) with i.i.d.\ weights.   These methods were subsequently applied to the exactly solvable planar directed last-passage percolation  (LPP)  model with i.i.d.\ exponential weights  by Ferrari and  Pimentel  \cite{buse3} and Coupier \cite{multicoupier}.  This model is also known as the exponential {\it corner growth model} (CGM).  

A key technical  point here is that the strict curvature hypotheses of Newman's work can be verified in the exactly solvable LPP model.  A second key feature is that the exponential LPP model can be coupled with the totally asymmetric simple exclusion process (TASEP).  This connection provides another suite of powerful tools for analyzing exponential LPP. 

 The  work of  \cite{multicoupier} and  \cite{buse3}  established for the exponential LPP model that, almost surely  for a fixed direction, directed semi-infinite geodesics from each lattice point  are unique and they coalesce.    An alternative  approach to these results was recently developed by one of the authors  \cite{coalnew}, by utilizing properties of the increment-stationary LPP process.
 
 Once coalescence is known, attention turns to quantifying it: how fast do semi-infinite geodesics started from two distinct points coalesce?    The scaling properties of planar models in the Kardar-Parisi-Zhang (KPZ) class come into the picture here.   This class consists of interacting particle systems, random growth models and directed polymer models in two dimensions (one of which can be time) that share universal fluctuation  exponents and limit distributions from random matrix theory. For surveys of the field, see \cite{introkpz1, introkpz2}.

 It is expected that, subject to mild moment assumptions on the weights, planar FPP and LPP are members of the KPZ class.  
   It is conjectured in general and proved in exactly solvable cases that a geodesic of length $N$ fluctuates on the scale $N^{2/3}$.  Thus if two semi-infinite geodesics start at distance $k$ apart, we expect coalescence to happen on the scale $k^{3/2}$.  
   
The first step in the study of the coalescence exponent was taken by W\"{u}thrich \cite{Wut-02}. 
  He proved a lower bound with 
exponent $3/2-\epsilon$  for LPP on planar Poisson points.  This was the first application of   the  first-passage
percolation techniques 
of  Newman and coauthors  in the context of an exactly solvable
last-passage percolation model.
The second step in this direction was taken by Pimentel  \cite{dual} for  the exponential  CGM.  By relying on the TASEP connection, he proved that in a fixed direction,  the so-called dual geodesic graph is equal in distribution (modulo a lattice reflection) to the original geodesic tree.   Next, by appeal to fluctuation bounds derived by coupling techniques in  \cite{cuberoot}, he derived an asymptotic  lower bound on the coalescence time, with the expected exponent $3/2$.  
   
 The next step taken by    Basu, Sarkar, and Sly  \cite{ubcoal} utilized the considerably more powerful  estimates   from integrable probability.  For the upper bound  on the coalescence time, they established not only the correct order of magnitude $k^{3/2}$  but also upper and lower probability bounds of matching orders of magnitude. In the same paper the original estimate of Pimentel was also improved significantly.
  


{\it Our goal in taking up the speed of  coalescence is the development of proof techniques that rely only on  the stationary version of the model and avoid both  the TASEP connection and  integrable probability.}    The applicability of this approach then covers all 1+1 dimensional KPZ models with a tractable stationary version.  This includes not only  various last-passage models in both discrete and continuous space, but also  the four currently known solvable positive temperature polymer models \cite{chau-noac-18-ejp}.   Extension beyond solvable models may also be possible,  as indicated by the  exact KPZ fluctuation exponents derived in \cite{Bal-Kom-Sep-12-aihp}  for a class of  zero-range processes outside currently known exactly solvable models.    This is work left for the future.    Another somewhat  philosophical point is that capturing exponents should be possible without integrable probability.  This has been demonstrated for fluctuation exponents by   \cite{cuberoot} for the exponential LPP and by \cite{poly2} for a positive-temperature directed polymer model.  

The results of this paper come from a unified approach  based on controlling   the exit point of the geodesic in a stationary LPP process and on Pimentel's duality of geodesics and dual geodesics.  This involves coupling, random walk estimates, planar monotonicity, and distributional properties of the stationary LPP process.   
  Here are the precise contributions of the present paper (details in Section \ref{s:main-th}): 
\begin{enumerate}[(i)] \itemsep=1pt
\item The upper and lower  bounds for slow coalescence  originally  due to Basu et al.~\cite{ubcoal}, though our   upper bound falls short of the optimal order by a logarithmic factor  (Theorem \ref{t:main-small}).  Our contribution here is to give a proof without integrable probability inputs.  
\item Upper and lower  bounds for fast coalescence of matching exponential order (Theorem \ref{t:main-large}). These are new results.  
\item A   lower bound on the transversal  fluctuations of a directed semi-infinite geodesic which improves bounds obtainable   without integrable probability (Theorem \ref{fluc}). 
\item  Strengthened exit time estimates for the stationary LPP process without integrable probability, some uniform over endpoints beyond a given distance (Theorems \ref{t:large-ub}, \ref{t:small-lb}, \ref{t:small-ub}). 
\end{enumerate}
%

We mention two more general but related points about the exponential CGM.  

(a) When all directions are considered simultaneously, the overall  picture of semi-infinite geodesics is richer than the simple almost-sure-uniqueness-plus-coalescence valid for  a fixed direction.  Part of this was already explained by Coupier \cite{multicoupier}. Recently the global picture  of uniqueness and coalescence was captured in \cite{geogeodesic}.
 Coalescence bounds that go beyond the almost surely unique geodesics in a fixed direction are   left as an open problem for the future. 

(b)  Various geometric  features of the   exponential LPP process can now be proved without appeal to properties of TASEP.  An exception is a deep result of Coupier \cite{multicoupier}  on the absence of triple geodesics in any random direction.  This fact  currently has no proof except the original one that relies on  the {\it TASEP speed process} introduced in \cite{tasepspeed}.


\subsubsection*{Organization of the paper} 
Precise definition of the exponential LPP model and the main results appear in Section \ref{s:main}.    Section \ref{s:prelim} collects known facts about  the CGM used in the proofs. This includes   properties of the stationary growth process and the construction of the directed semi-infinite geodesics in terms of Busemann functions.  Section \ref{s:exit-pf} derives new exit time estimates for the geodesic of the stationary growth process, stated as Theorems \ref{t:large-ub}, \ref{t:small-lb}, and \ref{t:small-ub}.  In the final Section \ref{s:proofs} the exit time estimates and duality are combined to prove the main results of Section \ref{s:main}.  The appendix contains a random walk estimate and a moment bound on the Radon-Nikodym derivative between two product-form exponential distributions. 

\subsubsection*{Notation and conventions} 
 Points $x=(x_1,x_2),y=(y_1,y_2)\in\R^2$ are ordered coordinatewise: $x\le y$ iff  $x_1\le y_1$  and $x_2\le y_2$.    The $\ell^1$ norm is $\abs{x}_1=\abs{x_1}+\abs{x_2}$.  The origin of $\R^2$ is denoted by both $0$ and $(0,0)$.  The two standard basis vectors are $\evec_1=(1,0)$ and  $\evec_2=(0,1)$. 
 For $a\le b$  in $\Z^2 $,  $\lzb a, b\rzb =\{x\in\Z^2: a\le x\le b\}$ is the rectangle in $\Z^2$ with corners $a$ and $b$.  $\lzb a, b\rzb$ is a segment if $a$ and $b$ are on the same horizontal or vertical line. 
 We use $\lzb a - e_1, a\rzb $, $\lzb a - e_2, a\rzb $ to denote unit edges when it is clear from the context.
 Subscripts indicate restricted subsets of the reals and integers: for example 
 $\Z_{>0}=\{1,2,3,\dotsc\}$ and $\Z_{>0}^2=(\Z_{>0})^2$ is the positive first quadrant of the planar integer lattice.    
   For $0<\alpha<\infty$, $X\sim$ Exp$(\alpha)$ means that the random variable $X$ has exponential distribution with rate $\alpha$, in other words $P(X>t)=e^{-\alpha t}$ for $t>0$ and $E(X)=\alpha^{-1}$.  
 
\vspace{5mm}
\noindent\textbf{Acknowledgments.} The authors gratefully thank Manan Bhatia for pointing out a mistake in the proof of Theorem 4.1 in the previous published  version of this paper. The authors also would like to thank the anonymous referee for his/her
suggestions about improving the exposition of this paper.

\section{Main results} 
\label{s:main}

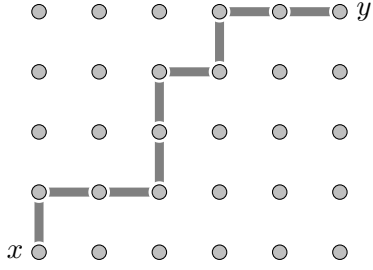
\begin{figure}
\captionsetup{width=0.8\textwidth}
\begin{center}

\begin{tikzpicture}[>=latex, scale=0.8]

			\draw[line width = 1.2mm, color=gray](0,0)--(0,1)--(2,1)--(2,3)--(3,3)--(3,4)--(5,4);
			
			\foreach \x in { 0,...,5}{
				\foreach \y in {0,...,4}{
					\fill[color=white] (\x,\y)circle(1.8mm); 
					\draw[ fill=lightgray](\x,\y)circle(1.2mm);

				}
			}

		\node at (-0.4,0) {$x$};
		\node at (5.4,4) {$y$};
			
		\end{tikzpicture}
	\end{center}	
	\caption{\small An up-right path between two integer points $x$ and $y$.}
	\label{fig_1}
\end{figure}

\subsection{The corner growth model and semi-infinite geodesics} 
The \textit{standard exponential corner growth model}  (CGM) is defined on the planar integer lattice  $\mathbb{Z}^2$ through  independent and identically distributed  (i.i.d.) weights  $\{\omega_{z}\}_{z\in \Z^2}$, indexed by the vertices of $\Z^2$, with marginal distribution $\w_z\sim$ Exp(1).  
  The {\it last-passage value} $G_{x,y}$ between two coordinatewise-ordered vertices $x\le y$ of $\Z^2$ is the maximal  total  weight of an  up-right nearest-neighbor path from $x$ to $y$: 
\beq\label{sec2G}G_{x,y} = \max_{z_{\bbullet}\, \in\, \Pi^{x,y}} \sum_{k=0}^{|y-x|_1} \omega_{z_k}\eeq
where $\Pi^{x,y}$ is the set of paths $z_{\bbullet}=(z_k)_{k=0}^{|y-x|_1}$ that satisfy $z_0=x$, $z_{|y-x|_1}=y$, and $z_{k+1}-z_k\in\{e_1, e_2\}$.   
The almost surely unique   maximizing path is  the point-to-point \textit{geodesic}.  
 $G_{x,y}$ is also called (directed) {\it last-passage percolation} (LPP). 
  If $x\le y$ fails our convention is  $G_{x,y} = -\infty$.  
 
A semi-infinite up-right path $(z_i)_{i=0}^\infty$ is a \textit{semi-infinite geodesic} if it is the maximizing path between any two points on this path, that is,  
$$\forall k<l \text{ in } \Z_{\geq0}:  \ (z_i)_{i=k}^l \in \Pi^{z_k, z_l} \quad \text{and}\quad  G_{z_k, z_l}= \sum_{i=k}^l \omega_{z_i} .$$
For a point $\xi\in\R^2_{\ge0}\setminus\{0\}$, the semi-infinite path $(z_i)_{i=0}^\infty$ is $\xi$-\textit{directed} if $z_i/|z_i|_1 \rightarrow \xi/|\xi|_1$ as $i\rightarrow \infty$. 

In the exponential CGM it is natural to index 
 spatial directions $\xi$  by a  real parameter  $\rho\in(0,1)$  through the equation  
\be\label{char1}
    \xi[\rho] = \left((1-\rho)^2, \rho^2 \right) . 
\ee
We call  $\xi[\rho]$  the {\it characteristic direction} associated to parameter $\rho$. This notion acquires meaning when we discuss the stationary LPP process in Section \ref{s:prelim}.  
Throughout, $N$ will be a scaling parameter that goes to infinity.   
When $\rho$ is understood, we write
\be\label{char1a}
    v_N = \left(\floor{N(1-\rho)^2}, \floor{N\rho^2} \right) 
\ee
for the lattice point moving in direction $\xi[\rho]$. 

The theorem below summarizes the key facts about directed semi-infinite geodesics that set the stage for our paper.  It goes back to the work of Ferrari and Pimentel \cite{buse3}  and Coupier \cite{multicoupier} on the CGM, and the general geodesic techniques introduced by Newman and coworkers \cite{How-New-97, How-New-01, fppcoal2, buse2}.  A different proof is given in \cite{coalnew}.  

\begin{theorem}\label{t:geod1}  Fix $\rho\in(0,1)$.  Then the following holds almost surely.  For each $x\in\Z^2$ there is a unique $\xi[\rho]$-directed semi-infinite geodesic $\bgeod{\,\rho}{x} =  \left(\bgeod{\,\rho}{x}_i\right)_{i=0}^\infty$ such that  $\bgeod{\,\rho}{x}_0=x$.   For each pair $x,y\in\Z^2$, the geodesics coalesce:  there is a coalescence point $\coal^\rho(x,y)$ such that 
$\bgeod{\,\rho}{x} \cap \bgeod{\,\rho}{y}  = \bgeod{\,\rho}{z} $ for $z=\coal^\rho(x,y)$. 
\end{theorem}

\begin{figure}[t]
\captionsetup{width=0.8\textwidth}
\begin{center}
 
\begin{tikzpicture}[>=latex, scale=1]

\draw[color=lightgray,line width = 0.3mm, ->] (0,-1)--(0,4);
\draw[color=lightgray,line width = 0.3mm, ->] (0,-1)--(5,-1);

\draw[lightgray,  line width = 0.5mm ]  (0,2) -- (3.5,2)--(3.5,-1);

\draw[loosely dotted, line width = 0.3mm, ->] (0,-1) -- (4.5,2+3/3.5);

\draw[ fill=black] (3,3)circle(1.2mm);

\draw[color=gray, dotted, line width = 1.2mm, ->] (0,0.3)-- (0.5,0.3) --(0.5,1)-- (1,1) -- (1,1.5)--(2,1.5)--(2.5,1.5) -- (2.5,3) -- (4,3) -- (4,4) ;

\draw[color=gray, dotted, line width = 1.2mm, ] (1.3,-1)-- (1.3,-0.5) --(1.8,-0.5) -- (1.8, 0) -- (2.3,0) -- (2.3,1) -- (3,1) -- (3,3);

\fill[color=white] (0,-1)circle(1.7mm); 
\draw[ fill=lightgray](0,-1)circle(1mm);
\node at (-0.7,0-1) {$(0,0)$};

\draw[ fill=black] (3,3)circle(1.2mm);
\draw[ fill=white](3,3)circle(0.7mm);

\fill[color=white] (3.5,2)circle(1.7mm); 
\draw[ fill=lightgray](3.5,2)circle(1mm);
\node at (4.1,2) {$v_N$};

\draw[color=lightgray,line width = 0.3mm, ->] (0+7,-1)--(0+7,4);
\draw[color=lightgray,line width = 0.3mm, ->] (0+7,-1)--(5+7,-1);

\draw[lightgray,  line width = 0.5mm ]  (7,2) -- (3.5+7,2)--(3.5+7,-1);

\draw[loosely dotted, line width = 0.3mm, ->] (0+7,-1) -- (4.5+7,2+3/3.5);

\draw[color=gray, dotted, line width = 1.2mm, ->] (0+7,-0.3) --(0.5+7, -0.3)-- (0.5+7,0.2) --(1+7,0.2)-- (1+7,0.5) --(1+7,1) -- (1+7,1) -- (1.5+7 ,1)-- (1.5+7,1.5)--(2+7,1.5)--(2.5+7,1.5) -- (2.5+7,2)  -- (2.5+7,3) -- (3+7+1,3) -- (3+7+1,4);
\draw[color=gray, dotted, line width = 1.2mm,] (0.7+7,-1) -- (8, -1) -- (8, -0.5) -- (8.5, -0.5) -- (8.5, 0.7) -- (9, 0.7) -- (9,1.5);
\fill[color=white] (0+7,-1)circle(1.7mm); 
\draw[ fill=lightgray](0+7,-1)circle(1mm);
\node at (-0.7+7,0-1) { $(0,0)$};

\draw[ fill=black] (2+7,1.5)circle(1.2mm);
\draw[ fill=white](2+7,1.5)circle(0.7mm);

\fill[color=white] (3.5+7,2)circle(1.7mm); 
\draw[ fill=lightgray](3.5+7,2)circle(1mm);
\node at (4.1+7,2) {$v_N$};

\draw[ fill=lightgray](0, 0.3)circle(1mm);

\node at (1.3, -1.4) {\footnotesize${rN^{2/3}}$};

\node at (-0.72, 0.3) {\footnotesize${rN^{2/3}}$};
 
\draw[ fill=lightgray](1.3, -1)circle(1mm);

\draw[ fill=lightgray](7, -0.3)circle(1mm);

\draw[ fill=lightgray](0.7+7, -1)circle(1mm);
\node at (0.8+7, -1.4) {\footnotesize${\delta N^{2/3}}$};

\node at (-0.75+7, -0.3) {\footnotesize${\delta N^{2/3}}$};

\end{tikzpicture}

\end{center}
\caption{\small Coalescence of $\xi[\rho]$-directed semi-infinite geodesics. The black circle marks the coalescence point: on the left it is $\coal^\rho(\floor{rN^{2/3}}e_1, \floor{rN^{2/3}}e_2)$, and on the right $\coal^\rho(\floor{\delta N^{2/3}}e_1, \floor{\delta N^{2/3}}e_2)$.   On the left for large $r$ the geodesics are likely to coalesce outside the rectangle $\lzb  0, v_N\rzb $,  while on the right for small $\delta$  the geodesics are likely to coalesce inside the rectangle $\lzb  0, v_N\rzb $.  }
\label{sec2fig1}
\end{figure}

\subsection{Coalescence estimates for  semi-infinite geodesics in a fixed direction} 
\label{s:main-th} 
 
  The two main results below give upper and lower bounds on the probability that two $\xi[\rho]$-directed  semi-infinite geodesics initially separated by a distance of order $N^{2/3}$  coalesce inside  the rectangle $\lzb0,v_N\rzb$.  The theorems are separated according to whether   the starting points of the geodesics are close to each other or far apart  on the scale $N^{2/3}$.    See the illustration in  Figure \ref{sec2fig1}. 
As introduced in Theorem \ref{t:geod1},  $\coal^\rho(x,y)$ is the coalescence point of the geodesics $\bgeod{\,\rho}{x}$ and  $\bgeod{\,\rho}{y}$.

\begin{theorem}\label{t:main-small}
For each  $0<\rho<1$ 
there exist finite positive constants 
$\delta_0$, $C_1$,  $C_2$ and  $N_0$ that depend only on $\rho$ and for which the following holds: 
whenever $N\geq N_0$  and $ N^{-2/3}\le\delta\le \delta_0$, 
\beq \label{1} C_1\delta \leq \mathbb{P} \big\{\coal^\rho(\floor{\delta N^{2/3}}e_1 , \floor{\delta N^{2/3}}e_2) \not\in \lzb 0, v_N\rzb \big\} \leq C_2 \abs{\log\delta\hspace{0.7pt}}^{2/3}\delta.\eeq
\end{theorem}
The requirement $\delta\ge N^{-2/3}$ in Theorem \ref{t:main-small} is needed only for the lower bound and only to ensure that $\floor{\delta N^{2/3}}\neq 0$.

\begin{theorem}\label{t:main-large}
For each  $0<\rho<1$ there exist finite positive constants $r_0, C_1, C_2$ and  $N_0$ that depend only on $\rho$ and for which the following holds: whenever  $N \geq N_0$ and $r_0\leq r \leq ((1-\rho)^2\wedge\rho^2) N^{1/3} $,
\beq e^{-C_1r^3} \leq \mathbb{P} \big\{\coal^\rho(\floor{r N^{2/3}}e_1, \floor{r N^{2/3}}e_2) \in \lzb 0, v_N\rzb \big\} \leq e^{-C_2 r^3}.\eeq
\end{theorem}

The requirement $ r \leq ((1-\rho)^2\wedge\rho^2) N^{1/3} $ in Theorem \ref{t:main-large} is needed only for the lower bound and only   to ensure that both geodesics start inside the rectangle $\lzb 0,v_N\rzb $.
  
If we replace one of the starting points with the origin $0$, the upper  bound of  Theorem \ref{t:main-small} and the lower  bound of Theorem \ref{t:main-large} hold automatically because   $\bgeod{\,\rho}{0}$ stays between  $\bgeod{\,\rho}{(\floor{r N^{2/3}},0)}$ and $\bgeod{\,\rho}{(0,\floor{r N^{2/3}})}$. The following corollary states that the other two tail estimates also hold with possibly different constants under this alteration in the geometry. 

\begin{corollary}\label{samelevel}
For each $0< \rho <1$ there exist finite positive constants $\delta_0, r_0, C_1, C_2$ and $N_0$ that depend only on $\rho$ and for which the following holds: whenever $N \geq N_0$, $N^{-2/3}\leq \delta \leq \delta_0$, and  $r\geq r_0$, 
\begin{enumerate}[{\rm(i)}]
\item $\mathbb{P} \big\{\coal^\rho(0, \floor{\delta N^{2/3}}e_1)  \not\in \lzb 0, v_N\rzb \big\} \geq C_1\delta$ and 
\item $\mathbb{P} \big\{\coal^\rho(0,\floor{r N^{2/3}}e_1) \in \lzb 0, v_N\rzb \big\} \leq e^{-C_2 r^3}.$
\end{enumerate}
\end{corollary}
  
\begin{remark} \label{rmkendpt}  Two comments about the results. 

(a)  The  statements of the theorems are valid for $v_N=(\fl{Na}, \fl{Nb})$ for any fixed $a,b>0$,  with new constants that depend also on $a,b$.  The characteristic point  $v_N$ of \eqref{char1a}  is simply one  natural choice.  

(b) The constants in the theorems that depend on $\rho\in(0,1)$ can be taken fixed uniformly for all $\rho$ in any compact subset of $(0,1)$.  
\end{remark}

 For  direct comparison with   \cite{ubcoal}, we state   two corollaries  for geodesics whose  locations are not expressed in terms of the large parameter $N$.

\begin{corollary}\label{c:main-small}
For each  $0<\rho<1$ 
there exist finite positive constants $R_0$, $C_1$ and  $C_2$   that depend only on $\rho$ and for which the following holds:   whenever $k\ge 1$ and $R\ge R_0$, 
\beq C_1 R^{-2/3} \leq \mathbb{P} \big\{\coal^\rho(\floor{k^{2/3}}e_1,\floor{k^{2/3}}e_2) \not\in\lzb 0, v_{Rk}\rzb \big\} \leq C_2(\log R)^{2/3}R^{-2/3}.\eeq
\end{corollary}

Corollary \ref{c:main-small} is derived from Theorem \ref{t:main-small} as follows. 
 Set $R_0= N_0\vee \delta_0^{-3/2}$.   Given $k\ge 1$  and $R\ge R_0$,   let  $N=Rk\ge N_0$ and $\delta=R^{-2/3} \le  \delta_0$.  
   Now $k^{2/3}=\delta N^{2/3}$. The next Corollary \ref{c:main-large} below is derived from Theorem \ref{t:main-large} in a similar way.

\begin{corollary}\label{c:main-large}
For each  $0<\rho<1$ 
there exist finite positive constants $R_1$, $C_1$ and  $C_2$ that depend only on $\rho$ and for which the following holds: whenever $k\ge 1$ and   $((1-\rho)^2\wedge\rho^2)^{-1} k^{-1/3}\le R\le R_1$, 
\beq e^{-C_1R^{-2} } \leq \mathbb{P} \big\{\coal^\rho(\floor{k^{2/3}}e_1 ,\floor{k^{2/3}}e_2)  \in \lzb 0, v_{Rk}\rzb\big\} \leq e^{-C_2 R^{-2} }.\eeq
\end{corollary}

 Again, the   lower bound $R\ge((1-\rho)^2\wedge\rho^2)^{-1} k^{-1/3}$ is imposed only to ensure that both  geodesics start inside the rectangle $\lzb 0,v_{Rk}\rzb $, for otherwise the probability in Corollary \ref{c:main-large} is zero.

  
The lower bounds in Theorem \ref{t:main-small}  and  Corollary \ref{c:main-small} are optimal, but the upper bounds are not due to the logarithmic factor.  Optimal upper and lower bounds (both of order $R^{-2/3}$) were proved for  Corollary \ref{c:main-small} by Basu, Sarkar, and Sly \cite{ubcoal} with   inputs from integrable probability.  Thus in Theorem \ref{t:main-small} and  Corollary \ref{c:main-small} our contribution is to provide bounds without relying on integrable probability.

Both upper and lower bounds in Theorem \ref{t:main-large} are new. 
The  upper bound $e^{-C_2r^3}$ of  Theorem \ref{t:main-large} improves significantly Pimentel's \cite{dual} asymptotic ($N\to\infty$)  upper bound   $Cr^{-3}$.   The improved bound  comes   from duality and an exit time estimate with the optimal exponential order,  obtained recently   by Emrah, Janjigian, and one of the authors in \cite{rtail1} without integrable probability inputs.  This exit time estimate was also derived independently by  Bhatia \cite{rtail2} with integrable probability inputs.     
In the intervening period between Pimentel's work and the present paper, Pimentel's bound   was improved to $e^{-Cr^{3/2}}$  (without sending $N$ to infinity)  in \cite{ubcoal} with inputs from integrable probability, see \cite[Remark 6.5]{ubcoal}.  
 



It is by now well-known that over distances of order $N$, geodesics fluctuate on the scale $N^{2/3}$.  A by-product of our proof is the following   lower bound on the size of the transversal fluctuation  of a semi-infinite geodesic. It is an improvement over previous bounds obtained without integrable probability (see Theorem 5.3(b) in \cite{CGMlecture}). 
 
\begin{theorem} \label{fluc}  For each $0<\rho <1$ there exist positive constants $C$, $N_0$ and $\delta_0$ that depend only on $\rho$ for which the following holds: whenever $N\geq N_0$ and $0<\delta \leq \delta_0$, 
\beq
\label{2} \mathbb{P}\bigl\{\bgeod{\,\rho}{(0,0)} \text{ enters the rectangle  
$\lzb  v_N - \delta N^{2/3}(e_1+e_2), v_N \rzb $
}\bigr\} \leq C\abs{\log\delta\hspace{0.5pt}}^{2/3}\delta.
\eeq
\end{theorem}
The  proofs in Section \ref{s:proofs} show that the probability in \eqref{2} is  essentially bounded above by the probability in \eqref{1}.  
With inputs from integrable probability, the upper bound $\abs{\log\delta\hspace{0.5pt}}^{2/3}\delta$ in \eqref{2} can be improved to $\delta$, the optimal upper bound for \eqref{1} obtained in \cite{ubcoal}. 


We turn to develop the groundwork for the proofs. 


%

\section{Preliminaries on the corner growth model} \label{s:prelim} 

This section covers aspects of the CGM used in the proofs. We provide illustrations, some intuitive arguments, and references to precise proofs.   The two main results are a fluctuation upper bound for the exit point of a stationary LPP process (Theorem \ref{t:exit1}) and the construction of semi-infinite geodesics with Busemann functions (Theorem \ref{t:buse}).  These are proved in article \cite{rtail1} and lecture notes  \cite{CGMlecture}, without using anything beyond the stationary LPP process. 


\subsection{Nonrandom properties} 

We begin with two basic features of LPP that   involve increments.  We state them for our exponential case but in fact these properties do not need any probability.  
  Let $G_{x,\hspace{0.5pt}\bbullet}$ be defined by \eqref{sec2G} and 
define    increment variables for $a\ge x+e_1$ and $b\ge x+e_2$  by 
$${I}^x_{ a} =  G_{x,a} - G_{x,a-e_1}\quad\text{and}\quad  {J}^x_{b} =  G_{x,b} - G_{x,b-e_2}.$$

The first property  is a monotonicity valid for  planar LPP.  Proof can be found for example 
in  Lemma 4.6 of \cite{CGMlecture}.  

\begin{lemma}\label{sec3lem1} For $y$ such that the increments are well-defined, 
$$  I^{x-e_1}_{ y}\le I^x_{ y} \leq   I^{x-e_2}_{y} \quad\text{and} \quad J^{x-e_2}_{y}\le  J^x_{ y} \leq   J^{x-e_1}_{y}.$$
\end{lemma}

\begin{figure}[t]
\captionsetup{width=0.8\textwidth}
\begin{center}
\begin{tikzpicture}[>=latex, scale=0.8]

\draw[line width=0.3mm, ->] (0,0) -- (8,0);
\draw[line width=0.3mm, ->] (0,0) -- (0,6);
\draw[gray ,dotted, line width=0.7mm] (0,0) -- (2,0) -- (2,0.5) -- (4, 0.5) -- (4,3) ;

\draw[lightgray, line width=0.5mm, ->] (3,3) -- (8,3);

\draw [decorate,decoration={brace,amplitude=10pt , mirror}, xshift=0pt,yshift=0pt]
(3,2.8) -- (7.8,2.8) ;
\node at (5.4,2.15) {$I^x_{z\,+\,\dbullet \, e_1}$};

\draw[lightgray, line width=0.5mm, ->] (3,3) -- (3,6);
\draw [decorate,decoration={brace,amplitude=10pt}, xshift=0pt,yshift=0pt]
(2.8,3) -- (2.8,5.8) ;
\node at (1.6,4.5) {$J^x_{ z\,+\,\dbullet \,e_2}$};

\draw[darkgray, dotted, line width=1.2mm] (3,3) -- (4,3); 
\draw[gray, dotted, line width=1.2mm]  (4,3) -- (4,5) -- (7,5);
\fill[color=white] (3,3)circle(1.7mm);
\draw[ fill=white](3,3)circle(1mm);
\node at (2.7,2.7) {$z$};

\draw[ fill=white](0,0)circle(1mm);
\node at (-0.3,-0.3) {$x$};

\fill[color=white] (4,3)circle(1.7mm); 
\draw[ fill=lightgray](4,3)circle(1mm);
\node at (4.3,3.3) {$a$};

\fill[color=white] (7,5)circle(1.7mm);
\draw[ fill=lightgray](7,5)circle(1mm);
\node at (7.3,5.3) {$y$};
\end{tikzpicture}
\end{center}
\caption{\small Illustration of Lemma \ref{sec3lem2a}. LPP process  $ G^{(x)}_{z,\bullet}$ uses boundary weights defined  by the LPP process $G_{x,\bullet}$.  Path $x$-$a$-$y$ is the geodesic of  $G_{x,y}$ and path $z$-$a$-$y$ the geodesic of  $G^{(x)}_{z,y}$. These geodesics share the segment $a$-$y$.}
\label{sec3fig2}
\end{figure}
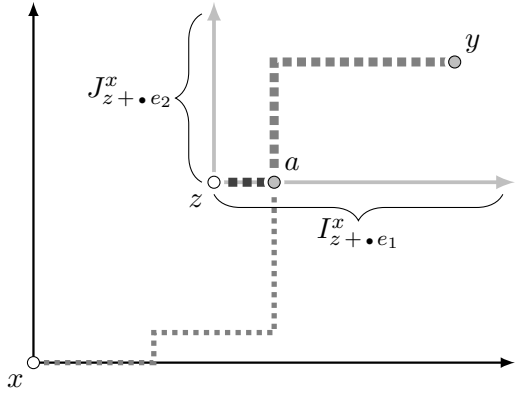

Fix distinct lattice points $x\le z$ and 
define a second LPP process $G^{(x)}_{z, \bbullet}$ with base point at $z$ that uses boundary weights given by the increments of $G_{x,\hspace{0.5pt}\bbullet}$, as illustrated in Figure \ref{sec3fig2}. Precisely,  for $y\ge z$, 
\beq\label{sec2G5}G^{(x)}_{z,y} = \max_{z_{\bbullet} \in \Pi^{z,y}} \sum_{k=0}^{|y-z|_1} \eta_{z_k}\eeq
where the weights are given by 
\be\label{eta4} \begin{aligned} 
\eta_z&=0,  
\qquad \eta_a=\w_a \quad\text{for }  a\in z+\Z_{>0}^2 \text{ (bulk), } \\
\eta_{z+ke_1}&= I^x_{z+ke_1}, \quad \eta_{z+ke_2}= J^x_{z+ke_1} \quad \text{for  } k\ge 1 \text{ (boundary). } 
\end{aligned}\ee
Proof of  the lemma below is elementary and   can be found in Lemma A.1 of \cite{CGMlecture}.

\begin{lemma} \label{sec3lem2a}   Let $x\le z$ and $y\in z+ \Z^2_{>0}$.  Then the  unique geodesics of  $G_{x,y}$ and $ G^{(x)}_{z,y}$ coincide in the quadrant $z+ \Z^2_{>0}$. 
\end{lemma}

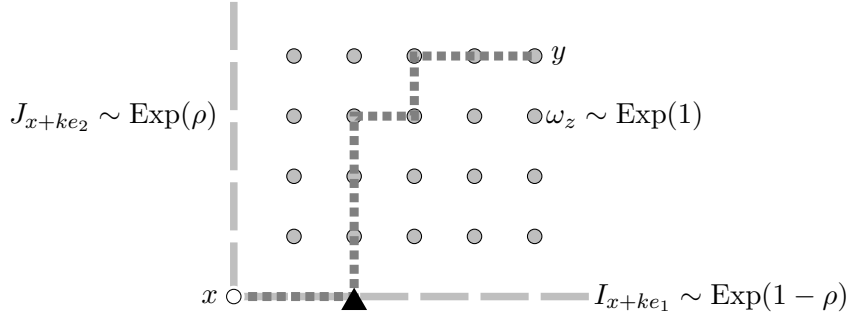
\begin{figure}[t]
\captionsetup{width=0.8\textwidth}
\begin{center}

\begin{tikzpicture}[>=latex, scale=0.8]

\draw[line width = 1mm, color = white] (3,6)--(4,6);
			
		
			\foreach \x in { 1,...,5}{
				\foreach \y in {1,...,4}{
					\fill[color=white] (\x,\y)circle(1.8mm); 
					\draw[ fill=lightgray](\x,\y)circle(1.2mm);
					}}

			\foreach \x in { 0,...,5}{
				\foreach \y in {0}{
					 
					\draw[line width = 1mm, color = lightgray] (\x+0.1,\y)--(\x+0.9,\y);

				}
			}
			
			\foreach \x in { 0}{
				\foreach \y in {0,...,4}{
					 
					\draw[line width = 1mm, color = lightgray] (\x,\y+0.1)--(\x,\y+0.9);

				}
			}
			
\draw[line width = 1.2mm, color=gray, dotted](0,0)--(2,0)--(2,1)--(2,3)--(3,3)--(3,4)--(5,4);

			\fill[color=white] (0,0)circle(1.8mm); 
					\draw[ fill=white](0,0)circle(1.2mm);

		\node at (8.1, -0) {$I_{x+ke_1} \sim \Exp(1-\rho)$};
		\node at (-2, 3) {$J_{x+ke_2} \sim \Exp(\rho)$};
		\node at (6.5,3) {$\omega_z \sim \Exp(1)$};
	
	\node at (-0.4,0) {$x$};
		\node at (5.4,4) {$y$};
\node[fill=,regular polygon, regular polygon sides=3,inner sep=2pt] at (2,-0.1) {};

		\end{tikzpicture}
	\end{center}	
	\caption{\small Increment-stationary LPP with base point $x$.  If the dotted line were the  geodesic of $G^\rho_{x,y}$, then the black triangle highlights the exit point, and the exit time is  $\exittime^{x\,\rightarrow\,y} = 2$.}
	\label{fig1stay}
\end{figure}

\subsection{Stationary last-passage percolation}

The stationary LPP process $G^\rho$  is defined on a positive quadrant $x+\Z^2_{\geq 0}$ with a  fixed base point  $x\in \Z^2$.  It is parametrized by $\rho\in(0,1)$.  Start with mutually independent bulk weights $\{\omega_z: z\in x+\Z_{>0}^2\}$ and boundary weights  $\{ I_{x+ke_1}, J_{x+le_2}:  k,l\in \Z_{>0}\}$ 
with marginal distributions
\beq\label{sec2weights} \omega_z \sim \Exp(1), \quad I_{x+ke_1} \sim \Exp(1-\rho),\quad\text{and}\quad  J_{ x+le_2} \sim \Exp(\rho).\eeq
The probability distribution of these weights is denoted by $\mathbb{P}^\rho$. 
The LPP process $G^\rho_{x,\hspace{0.5pt}\bbullet}$ is defined on the boundary of the quadrant by 
  $G^\rho_{x,x}  = 0$, 
$ G^\rho_{x,x+ke_1}  =  \sum_{i=1}^k I_{x+ie_1}$ and $ G^\rho_{x,x+le_2}  = \sum_{j=1}^l J_{x+je_2}$  for $k,l\ge 1$.  
In the bulk we perform LPP that  uses both the boundary and the  bulk weights:  for 
$y = x + (m,n)\in x+\Z_{>0}^2$, 
\begin{align}\label{sec2G^rho}
G^\rho_{x,y} = \max_{1\leq k \leq m} \biggl\{ \biggl(\;\sum_{i=1}^k I_{x+ie_1} \biggr)+ G_{x+ke_1+e_2, y}\biggr\}  
 \bigvee \max_{1\leq l \leq n} \biggl\{ \biggl( \;\sum_{j=1}^l J_{x+je_2} \biggr)+ G_{x+le_2+e_1, y}\biggr\}.
\end{align}
The LPP value $G_{a,b}$ inside the braces is the standard one defined by \eqref{sec2G} with the i.i.d.\ bulk weights $\w$. 
Call the almost surely unique maximizing path a \textit{$\rho$-geodesic}.   
The \textit{exit time} $\exittime^{x\,\to \,y}$ is the $\Z\setminus\{0\}$-valued random variable that records where the $\rho$-geodesic from $x$ to $y$ exits the boundary, relative to the base point $x$,  with a sign that indicates choice between the axes:
\begin{align}\label{sec2G^rho1}
G^\rho_{x,y} =  \begin{cases} \sum_{i=1}^{k}  I_{x+ie_1} + G_{x+ke_1+e_2, y}, &\text{ if } \exittime^{x\,\to \,y}=k>0 \\[4pt] 
\sum_{j=1}^{l} J_{x+je_2} + G_{x+le_2+e_1, y}, &\text{ if } \exittime^{x\,\to \,y}=-l<0.  
\end{cases} 
\end{align}
See Figure \ref{fig1stay} for an illustration.

Define  horizontal and vertical increments of $G^\rho_{x,\hspace{0.5pt}\bbullet}$   as
\beq \label{uselessdef} I^x_{a} = G^\rho_{x,a} - G^\rho_{x,a-e_1}
\quad\text{and}\quad
J^x_{ b} = G^\rho_{x,b} - G^\rho_{x,b-e_2}\eeq
for   $a \in x+\Z_{>0}\times\Z_{\geq 0}$ and  $b \in x+\Z^2_{\geq 0}\times\Z_{>0}$. 
The definition above implies $I^x_{ke_1}=I_{ke_1}$ and $J^x_{le_2}=J_{le_2}$ for $k,l\ge 1$.  The term ({\it increment}) {\it stationary} LPP is justified by the next fact. Its proof is an induction argument and can be found for example in \cite[Thm.~3.1]{CGMlecture}. 

\begin{lemma}\label{Grho-lm5}    Let $\{y_i\}$ be any finite or infinite down-right path in $x+\Z_{\ge0}^2$. That is,  $(y_{i+1}-y_i)\cdot e_2\le 0\le (y_{i+1}-y_i)\cdot e_1$.   Then the increments $\{ G^\rho_{x, y_{i+1}}-G^\rho_{x, y_i}\}$ are independent.  The marginal distributions of nearest-neighbor increments are $I^x_{a} \sim {\rm Exp}(1-\rho)$ and $J^x_{b} \sim {\rm Exp}(\rho)$.
\end{lemma} 

Now apply Lemma \ref{sec3lem2a} to this stationary situation.  Take $z\in x+\Z^2_{\geq 0}$ and define  the LPP process  $G^{(x),\rho}_{z,\,\bbullet}$ with the recipe \eqref{sec2G5} where the boundary weights   are the ones  in \eqref{uselessdef}.  By Lemma \ref{Grho-lm5}, these boundary weights have the same distribution as the original ones in \eqref{sec2weights}. Consequently $G^{(x),\rho}_{z,\bbullet}$ is another stationary LPP process. 
Lemma \ref{sec3lem2a} gives the statement below which will be used extensively in our proofs.


\begin{lemma} \label{sec3lem2}   Let $x\le z$ and $y\in z+ \Z^2_{>0}$.  Then the  unique geodesics of  $G^\rho_{x,y}$ and $ G^{(x),\rho}_{z,y}$ coincide in the quadrant $z+ \Z^2_{>0}$. 
\end{lemma}

 Since the boundary weights in \eqref{sec2weights}  are stochastically larger than the bulk weights, the $\rho$-geodesic prefers the boundaries. The \textit{characteristic direction} $\xi[\rho] = ((1-\rho)^2, \rho^2)$ defined earlier in \eqref{char1}  is the unique direction in which the attraction of the $e_1$- and $e_2$-axes      balance each other  out.   A consequence of this is that the $\rho$-geodesic from $x$ to $x+v_N$  spends order  $N^{2/3}$  steps on the boundary.  Here we encounter the $2/3$ wandering exponent of KPZ universality.   This is described in Theorems \ref{t:exit1} and \ref{t:small-ub} below.  The macroscopic picture is   in Figure \ref{chardir}.   This matter is discussed more thoroughly in Section 3.2 of \cite{CGMlecture}.   We record the upper bound for this exit time recently derived in \cite{rtail1}.

\begin{figure}[t]
\captionsetup{width=0.8\textwidth}
\begin{center}
\begin{tikzpicture}[scale = 0.8]

\draw[gray, line width=0.3mm, ->] (0,0) -- (5,0);
\draw[gray, line width=0.3mm, ->] (0,0) -- (0,4);

\draw[dotted, line width=1mm] (0,0) -- (0.6*5 , 0.6*4);

\fill[color=white] (0.6*5 ,0.6*4)circle(1.7mm); 
\draw[ fill=lightgray](0.6*5, 0.6*4)circle(1mm);

\fill[color=white] (0,0)circle(1.7mm); 
\draw[ fill=lightgray](0,0)circle(1mm);

\draw[gray, dotted, line width=0.8mm] (0,0) -- (1.4,0) -- (0.6*5+1.4 , 0.6*4);

\draw[gray, dotted, line width=0.8mm] (0,0) -- (0,1) -- (2.5,2+1);

\draw[gray ,dotted, line width=0.3mm, ->] (0,0) -- (5,4) ;

\node at (5, -0.4) {$\Exp(1-\rho)$};
\node at (-0.8, 4) {$\Exp(\rho)$};

\node at (5+0.4, 4-0.3) {$\xi[\rho]$};

\node at (0-0.7, -0.1) {$(0,0)$};

\fill[color=white] (0.6*5+1.4 ,0.6*4)circle(1.7mm); 

\draw[ fill=lightgray](0.6*5+1.4, 0.6*4)circle(1mm);

\fill[color=white] (2.5,2+1)circle(1.7mm); 
\draw[ fill=lightgray](2.5,2+1)circle(1mm);

\fill[color=white] (0,0)circle(1.7mm); 
\draw[ fill=white](0,0)circle(1mm);

\end{tikzpicture}
\end{center}
\caption{\small A macroscopic view of  point-to-point geodesics (dotted lines) in stationary LPP  from the base point at the origin $(0,0)$  to three different endpoints (gray bullets).  Only the geodesic in the characteristic direction $\xi[\rho]$ spends no macroscopic time on the boundary.}
	\label{chardir}
\end{figure}
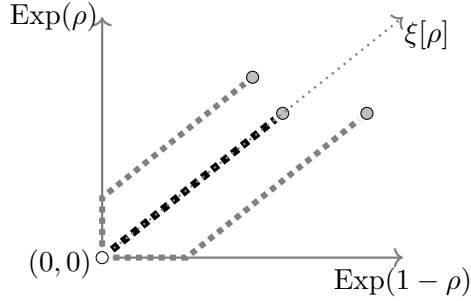

\begin{theorem}\label{t:exit1}   {\rm\cite[Theorem 2.5]{rtail1}}
  There exist positive constants $r_0, N_0$, $C$ that depend only on $\rho$ such that for all $r>r_0$, $N\geq N_0$,   and  $|v- v_N|_1\leq N^{2/3}$, 
$$\mathbb{P}^\rho \big\{|\exittime^{\,0\,\rightarrow \,v}|  \geq rN^{2/3}\big\} \leq e^{-Cr^3}.$$
\end{theorem}

In the next  corollary the $\Theta(N^{2/3})$  deviation is transferred from the base point $0$ to the endpoint $v_N$.   Figure \ref{sec3fig4} illustrates how Lemma \ref{sec3lem2} reduces claim  \eqref{weprove2}   to  Theorem \ref{t:exit1}.  
(Corollary \ref{sec3cor} is proved using the same method as Corollary 5.10 in the arXiv version of  \cite{CGMlecture}.) 


\begin{corollary}\label{sec3cor}
There exist positive constants $N_0$, $C$ that depend only on $\rho$ such that for $N \geq N_0$  and $b>0$,
\begin{align}
 \label{weprove}&\mathbb{P}^\rho \big\{\exittime^{\,0\,\rightarrow  \, v_N+ \floor{bN^{2/3}}e_1} \leq - 1 \big\} \leq e^{-Cb^3} \qquad\text{and} \\
\label{weprove2}&\mathbb{P}^\rho \big\{\exittime^{\,0\,\rightarrow  \, v_N- \floor{bN^{2/3}}e_1} \geq 1\big\} \leq e^{-Cb^3}.
 \end{align}
\end{corollary}

\begin{figure}[t]
\captionsetup{width=0.8\textwidth}
\begin{center}
 
\begin{tikzpicture}[>=latex, scale=1]

\draw[line width = 0.3mm, ->] (1,0)--(1,4);
\draw[line width = 0.3mm, ->] (1,0)--(5,0);

\draw[dotted, color=gray, line width = 1mm] (1,0)--(8.3-7,0) -- (8.3-7,0.5) --(9.5-7, 0.5) -- (9.5-7,2);

\fill[color=white] (2.5,2)circle(1.4mm);

\draw[ line width = 0.3mm, ->] (6.5,0)--(6.5,4);
\draw[line width = 0.3mm] (6.5,0)--(8,0);
\draw[line width = 0.8mm, ->] (8,0)--(12,0);
\draw[line width = 0.8mm, ->] (8,0)--(8,4);

\node at (5.4,2.8) {\small $\xi[\rho]$};

\draw[dotted, color=gray, line width = 1mm] (6.5,0)--(8.3,0) -- (8.3,0.5) --(9.5, 0.5) -- (9.5,2);

\draw[loosely dotted, line width = 0.5mm, ->] (1,0)--(5,2+2/3);

\draw[loosely dotted, line width = 0.5mm, ->] (6.5,0)--(6.5+0.8*5.5, 0.8*3.5);
\node at (6.5+0.8*5.9, 0.8*3.9) {\small $\xi[\rho]$};
\fill[color=white] (2.5,2)circle(1.7mm);
\draw[ fill=lightgray](2.5,2)circle(1mm);

\draw [decorate,decoration={brace,amplitude=10pt}, xshift=0pt,yshift=0pt]
(2.5,2+0.2) -- (4,2+0.2) ;

\node at (3.25+0.2,2.8) {\small $bN^{2/3}$};

\draw[ fill=lightgray](4,2)circle(1mm);
\node at (4.4,1.8) {\small $v_N$};

\draw[ fill=lightgray](1,0)circle(1mm);

\draw[ fill=lightgray](6.5,0)circle(1mm);

\draw[ fill=lightgray](11,2)circle(1mm);

\draw[ fill=lightgray](8,0)circle(1mm);



\draw [decorate,decoration={brace,amplitude=10pt, mirror}, xshift=0pt,yshift=0pt]
(6.5+3,-0.2+2) -- (8+3,-0.2+2) ;

\node at (7.5+3,-0.7+2) {\small $bN^{2/3}$};

\fill[color=white] (9.5,2)circle(1.7mm);
\draw[ fill=lightgray](9.5,2)circle(1mm);

\node at (0.4,0) {\small $(0,0)$};

\node at (8,-0.4) {\small $(0,0)$};

\node at (6.5,-0.4) {\small $-bN^{2/3}e_1$};

\end{tikzpicture}
 
\end{center}
\caption{\small  Proof of  \eqref{weprove2}. On the left the event $\exittime^{\,0\,\rightarrow  \, v_N- \floor{bN^{2/3}}e_1} \geq 1$.    On the right a second base point is placed at  $- \floor{bN^{2/3}}e_1$ and the increment variables on the $e_2$-axis based at $0$ are determined by the LPP process based at $- \floor{bN^{2/3}}e_1$. By Lemma  \ref{sec3lem2},  $\exittime^{\,0\,\rightarrow  \, v_N- \floor{bN^{2/3}}e_1} \geq 1$ iff $\exittime^{\,- \floor{bN^{2/3}}e_1\,\rightarrow  \, v_N- \floor{bN^{2/3}}e_1} \geq bN^{2/3}$. This last event has probability $\le e^{-Cb^{-3}}$ by Theorem \ref{t:exit1}.}
\label{sec3fig4}
\end{figure}


\subsection{Busemann functions and semi-infinite geodesics}\label{s:buse}

The key to our results is that the directed semi-infinite geodesics can be defined through Busemann functions, which themselves are instances of stationary LPP. Thus estimates proved for stationary LPP provide information about the behavior of directed semi-infinite geodesics.  

The next theorem summarizes the properties of Busemann functions needed.  It is a combination of results from Section 4 of \cite{CGMlecture} and Lemma 4.1 of \cite{coalnew}.  
The dual weights introduced in part (iii) below are connected with {\it dual geodesics} which will be constructed later in Section 5. 

\begin{theorem}\label{t:buse}   Fix $\rho\in(0,1)$.  Then on the probability space of the i.i.d.\ ${\rm Exp}(1)$ weights $\{\w_z\}_{z\in\Z^2}$ there exists a process $\{B^\rho_{x,y}\}_{x,y\in\Z^2}$ with the following properties. 
\begin{enumerate} [{\rm(i)}] 

\item  With probability one,  $  \forall x,y \in \Z^2$, 
$$B^\rho_{x,y} = \lim_{N\rightarrow \infty} \bigl( G_{x, u_N} - G_{y, u_N}\bigr) 
$$
for any sequence $u_N$ such that $\abs{u_N}\to\infty$ and  $u_N/|u_N|_1 \rightarrow \xi[\rho]/|\xi[\rho]|_1$ as $N\rightarrow \infty$.

\item The unique $\xi[\rho]$-directed semi-infinite geodesic from $x$ is defined by $\bgeod{\,\rho}{x}_0=x$ and for $k\ge 0$,  
\beq\label{busegeo}  
\bgeod{\,\rho}{x}_{k+1}=\begin{cases}   \bgeod{\,\rho}{x}_{k} + e_1, &\text{if } \ B^\rho_{\bgeod{\,\rho}{x}_{k},\,\bgeod{\,\rho}{x}_{k} + e_1} \le  B^\rho_{\bgeod{\,\rho}{x}_{k},\,\bgeod{\,\rho}{x}_{k} + e_2}
\\[5pt]  
 \bgeod{\,\rho}{x}_{k} + e_2, &\text{if } \   B^\rho_{\bgeod{\,\rho}{x}_{k},\,\bgeod{\,\rho}{x}_{k} + e_2} <  B^\rho_{\bgeod{\,\rho}{x}_{k},\,\bgeod{\,\rho}{x}_{k} + e_1}. 
\end{cases} 
  \eeq
\item  Define the  dual weights by 
$$\text{$\widecheck{\omega}^\rho_z = B^\rho_{z-e_1,z}\wedge B^\rho_{z-e_2,z}$ \ \ \ for $z\in \Z^2$.} $$ 
 Fix a  bi-infinite nearest-neighbor  down-right path $\gamma = \{x_i\}_{i\in\Z}$ on $\Z^2$.   
 This means that $x_{i+1}-x_i\in\{e_1,  -e_2\}$.  Then the random variables 
 \begin{align*}  &\{ B^\rho_{x_{i},x_{i+1}}: i\in\Z\}, \ 
  \{\omega_y: 
\text{$y\in\Z^2$ lies strictly to the left of and below $\gamma$} \}, \\
&\qquad \text{and} \quad 
 \{\widecheck\w^\rho_z: 
\text{$z\in\Z^2$ lies strictly to the right of and above $\gamma$} \}
\end{align*} 
  are all mutually independent with  marginal distributions 
\be\label{buse78}  B^\rho_{x, x+e_1}\sim {\rm Exp}(1-\rho),
\quad  
B^\rho_{x, x+e_2}\sim {\rm Exp}(\rho) \quad\text{and}\quad \omega_y, \;\widecheck{\omega}^\rho_z\sim {\rm Exp}(1).
\ee

\end{enumerate} 

\end{theorem}

Versions of parts (i) and (ii) above can be proved for  general i.i.d.\ weights  \cite{Geo-Ras-Sep-17-ptrf-1}.  But nothing like part (iii) with precise distributions for  Busemann functions and  dual weights is known for LPP models that are not exactly solvable. 


A Busemann function $B^\rho$ can be thought as a stationary LPP process in two ways. One with north and east boundaries, denoted by $G^{\rho, NE}$, and one with south and west boundaries, denoted by  $G^{\rho}$. Here $G^\rho$ is as was given in  \eqref{sec2G^rho}, and $G^{\rho, NE}$ is defined as follows (NE stands for north and east boundaries). 

Fix an origin or base point $x\in \Z^2$.   Start with mutually independent bulk weights $\{\omega_z : z\in x-\Z^2_{>0}\}$ and boundary weights   $\{ I_{x-ke_1}, J_{x-le_2}:  k,l\in \Z_{\ge0}\}$ 
with marginal distributions
$$ \omega_z \sim \Exp(1), \quad I_{x-ke_1} \sim \Exp(1-\rho),\quad\text{and}\quad  J_{ x-le_2} \sim \Exp(\rho).$$
On the boundaries define  $G^{NE,\rho}_{x,x} = 0$, $ G^{NE,\rho}_{x-ke_1, x}  =  \sum_{i=0}^{k-1} I_{x-ie_1}$ and $ G^{NE,\rho}_{x+le_2, x}  = \sum_{j=0}^{l-1} J_{x-je_2}$  for $k,l\ge 1$.  
In the bulk we perform LPP that  uses both the boundary and the  bulk weights:  for 
$y = x  -  (m,n)\in x-\Z_{>0}^2$, 
\be\label{G-NE6} \begin{aligned}
G^{NE,\rho}_{y,x} = \max_{1\leq k \leq m} \biggl\{ \biggl(\;\sum_{i=0}^{k-1} I_{x-ie_1} \biggr)+ G_{y, x-ke_1-e_2}\biggr\}  
 \bigvee \max_{1\leq l \leq n} \biggl\{ \biggl( \;\sum_{j=0}^{l-1} J_{x-je_2} \biggr)+ G_{y, x-le_2-e_1}\biggr\}.
\end{aligned}\ee
The LPP value $G_{a,b}$ inside the braces is the  one defined by \eqref{sec2G} with i.i.d.\ bulk weights $\w$. 

  Two stationary LPP processes can be defined by taking Busemann increments as boundary weights.  Fix again a base point $x\in\Z^2$.
\begin{itemize}
\item Construct  $G^{\rho, NE}_{y,x}$ for  $y \leq x$ as in \eqref{G-NE6} using  the NE boundary weights 
$I_{x-ke_1}=B^\rho_{x-(k+1)e_1, x-ke_1}$ and  $J_{x-le_2}=B^\rho_{x-(l+1)e_2, x-le_2}$ 
and bulk weights $\{\w_z: z\in x-\Z_{>0}^2\}$. 
\item Construct  $G^\rho_{x, y'}$ for $y'\geq x$ as in \eqref{sec2G^rho} using   the SW boundary weights 
$I_{x+ke_1}=B^\rho_{x+(k-1)e_1, x+ke_1}$ and   $J_{x+le_2}=B^\rho_{x+(l-1)e_2,  x+le_2}$ and bulk weights $\{\wc\w^\rho_z: \in x+\Z_{>0}^2\}$. 
\end{itemize}
These two constructions satisfy the definitions of stationary LPP processes due to Theorem \ref{t:buse}(iii). Their key properties relative to the Busemann function are 
\begin{align}
G^{\rho, NE}_{y, x}  =  B^\rho_{y,x} \quad &\text{ for all $y \leq x$} \label{canbegeneral}\\
\text{and} \qquad G^{\rho}_{x, y'}  =  B^\rho_{x,y'} \quad &\text{ for all $y' \geq x$}. \label{needexact}
\end{align}
This is   in Theorem 4.4 of \cite{CGMlecture}. 

As the last point, we state  an independence property for a coupling of Busemann functions in two different directions. This fact was used to show the non-existence of bi-infinite geodesics \cite{balzs2019nonexistence} and local stationarity of the CGM \cite{balzs2020local}. It follows from the queuing map construction for the joint distribution (in various directions) of Busemann function from \cite{Fan-Sep-18-}.

\begin{proposition}{\rm\cite[Lemma 4.5]{balzs2020local}} \label{indbuse}
Let $0< \eta < \lambda < 1$. There exists a coupling of Busemann functions $B^\eta$ and $B^\lambda$  such that for any fixed $x\in \mathbb{Z}^2$ and  for every $k,l \in \Z_{>0}$, the following sets of random variables (on the horizontal line through $x$) are independent:  
$$
\big\{B^\eta_{x+ie_1,x+(i+1)e_1}\big\}_{-k\leq i\leq -1} \quad\text{and}\quad \big\{B^\lambda_{x+ie_1, x+(i+1)e_1}\big\}_{0\leq i \leq l-1}.
$$

\end{proposition}

\section{Exit time estimates} \label{s:exit-pf} 

This section proves estimates on the exit time for stationary LPP processes defined in \eqref{sec2G^rho} and \eqref{sec2G^rho1}.  
 These results   are applied in Section \ref{s:proofs} to prove the main theorems stated in Section \ref{s:main}.  
The first theorem is the main intermediate result towards the  lower bound of Theorem \ref{t:main-large}.  We also introduce  useful lemmas that are  used again later in the  proof of Theorem \ref{t:small-ub}. 
\begin{theorem} \label{t:large-ub} For each  $0<\rho<1$ 
there exist finite positive constants $r_0(\rho)$, $C(\rho)$ and  $N_0(\rho)$  such that 
for all $N\geq N_0(\rho)$ and  $r_0 \leq r \leq [(1-\rho)^2\wedge\rho^2] N^{1/3}$,
\[ \mathbb{P}^\rho \big\{ \forall z \text{ outside } \lzb  0, v_N \rzb  \text{ we have } |\exittime^{\,0\,\rightarrow \,z}| \geq r N^{2/3}\big\}   \geq  e^{-Cr^3} .\]
\end{theorem}

To prove this bound we tilt the probability measure to make the event likely and pay for this with a moment bound on   the Radon-Nikodym derivative. This argument was introduced in \cite{bala-sepp-aom} in the context of ASEP, and adapted to  a lower bound proof of the  longitudinal  fluctuation exponent in the stationary LPP in Section 5.5 of the lectures \cite{CGMlecture}. The key idea is a perturbation of the parameter $\rho$ of the stationary LPP process to   $\rho \pm rN^{-1/3}$. This allows us to control the exit point on the scale $N^{2/3}$.  The general  idea of utilizing perturbations of order $N^{-1/3}$ goes back to the seminal paper \cite{cato-groe-06}. 

Lemma \ref{lemmaab} below is  an auxiliary estimate for the proof of Theorem \ref{t:large-ub}.  It  utilizes a perturbed parameter $\lambda = \rho + rN^{-1/3}$, assumed 
to satisfy 
\beq \label{lambda} \rho < \lambda  \leq  c(\rho)  < 1 \eeq
for some constant $c(\rho)<1$, as $r$ and $N$ vary.  
   Lemma \ref{lemmaab} shows that, for   small enough $a>0$ and large enough $b, r>0$,  the $\lambda$-geodesic   to a target point $w_N$ slightly perturbed from $v_N$ exits the $e_1$-axis   through the interval $[\![ arN^{2/3}e_1, brN^{2/3}e_1]\!]$ with high probability.  This is illustrated  on the right of Figure \ref{fig99}.   
The constants $1-\rho$ and ${2}/{\rho^2}$ in Lemma \ref{lemmaab} come from the following observation (left diagram of Figure \ref{fig99}).   Start two rays at $(0,0)$ in the directions $\xi[\rho]$ and $\xi[\lambda]$ and let $u_N$ be the lattice point closest to the $\xi[\lambda]$-directed ray such that $u_N \cdot e_2 = v_N \cdot e_2$.  Then
\beq \label{cbound} (1-\rho) rN^{2/3} \leq v_N\cdot e_1 - u_N \cdot e_1 \leq \frac{2}{\rho^2}rN^{2/3}. \eeq

\begin{figure}[t]
\captionsetup{width=0.8\textwidth}
\begin{center}
\begin{tikzpicture}[scale = 0.8]

\draw[gray, line width=0.3mm, ->] (0,0) -- (5,0);
\draw[gray, line width=0.3mm, ->] (0,0) -- (0,4);

\draw[gray ,loosely dotted, line width=0.5mm, ->] (0,0) -- (1*5, 1*4);

\draw[gray ,loosely dotted, line width=0.5mm, ->] (0,0) -- (5-1.4/0.8, 4) ;
\fill[color=white] (0.8*5, 0.8*4)circle(1.7mm); 
\draw[ fill=lightgray](0.8*5, 0.8*4)circle(1mm);
\node at (0.8*5, 0.8*4+0.4) {$v_N$};

\fill[color=white] (0.8*5- 1.4, 0.8*4)circle(1.7mm); 
\draw[ fill=lightgray](0.8*5- 1.4, 0.8*4)circle(1mm);
\node at (0.8*5- 1.5 , 0.8*4+0.4) {$u_N$};

\node at (5+0.4, 4-0.3) {$\xi[\rho]$};
\node at (5-1.4/0.8+0.1, 4+0.4){$\xi[\lambda]$};

\node at (0-0.6, 0) {$(0,0)$};

\draw[black ,dotted, line width=0.3mm] (1.4 ,0) -- (0.8*5, 0.8*4);
\draw[ fill=black](1.4, 0)circle(1.3mm);

\draw[gray, line width=0.3mm, ->] (0+8,0) -- (5+8.5,0);
\draw[gray, line width=0.3mm, ->] (0+8,0) -- (0+8,4);

\draw[gray ,dotted, line width= 1mm] (8,0)--(9.1,0) --(9.1,0.6)--(9.7,0.6)--(9.7,0.8)--(9.7,1.1) --(10.1,1.1)--(10.1,1.5)--(10.9,1.5) -- (10.9,1.9)--(11.5,1.9)--(11.5,2.7)--(12,2.7)--(12,3.2);

\draw[ line width=1mm] (0.24+8.5,-0.1) -- (0.24+8.5,0.1);
\draw[ line width=1mm] (2+8.5,-0.1) -- (2+8.5,0.1);

\draw[black ,dotted, line width=0.3mm]  (1.4+8.5,0) -- (0.8*5+8.5, 0.8*4);
\draw[black ,dotted, line width=0.3mm] (1.4+8.5-0.5 ,0) -- (0.8*5+8.5-0.5, 0.8*4);


\draw[ fill=lightgray](0.8*5+8.5, 0.8*4)circle(1mm);
\node at (0.8*5+9, 0.8*4) {$v_N$};

\draw[ fill=lightgray](0.8*5+8.5-0.5, 0.8*4)circle(1mm);
\node at (0.8*5+8.5-0.6, 0.8*4+0.4) {$w_N$};

\node at (0.1+8.5, -0.4) {${arN^{2/3}}$};
\node at (2.5+8.5, -0.4) {${brN^{2/3}}$};

\node at (5+8.5, -0.4) {$\Exp(1-\lambda)$};
\node at (-0.8+8, 4) {$\Exp(\lambda)$};


\node at (0-0.6+8, 0) {$(0,0)$};

\draw[ fill=black](1.4+8.5, 0)circle(1.3mm);

\draw[ fill=white](0.9+8.5, 0)circle(1.3mm);

\end{tikzpicture}
\end{center}
\caption{\small \textit{Left:} Illustration of estimate \eqref{cbound}. \textit{Right:} Illustration of Lemma \ref{lemmaab}.  
The dotted lines have characteristic slope $\xi[\lambda]$. Consequently, with high probability, the geodesic from $0$ to $w_N$ exits through the interval $[\![ arN^{2/3}e_1, brN^{2/3}e_1]\!]$.}
\label{fig99}
\end{figure}
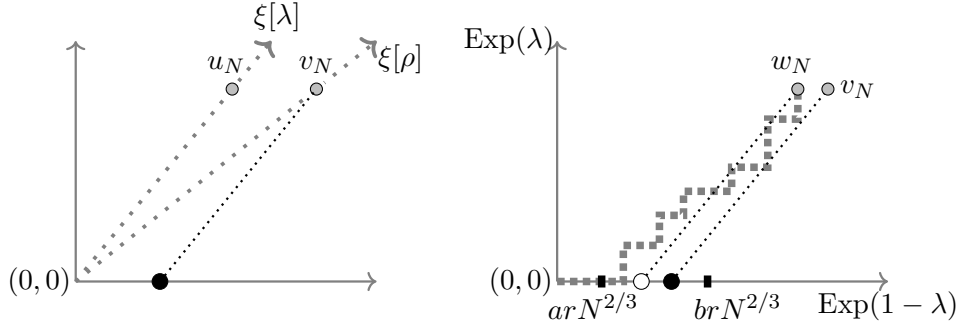

\begin{lemma} \label{lemmaab} 
Let $\lambda = \rho + rN^{-1/3}$ and $w_N = v_N - \floor{\frac{1}{10}(1-\rho)rN^{2/3}}e_1$. There exist positive constants $ C, N_0$ that  depend only on $\rho$ such that, for any $r>0$ and  $N \geq N_0$ such that \eqref{lambda} holds, we have 
$$\mathbb{P}^\lambda \left({\tfrac1{10}(1-\rho)rN^{2/3}} \leq \exittime^{\,0\,\rightarrow\, w_N} \leq {10\frac{2}{\rho^2}rN^{2/3}} \right) \geq 1- e^{-Cr^3}.$$
\end{lemma}

Before the proof of Lemma \ref{lemmaab}, 
we separate an  observation about geodesics in the next lemma, illustrated by the left diagram of Figure \ref{figlemmaab}. 
It  comes from the idea of Lemma  \ref{sec3lem2a} of constructing nested LPP processes with boundary weights defined by increments of an outer LPP process.    (Lemma \ref{relatetau} is proved as  Lemma A.3 in the appendix of  \cite{CGMlecture}.)

\begin{lemma}\label{relatetau}  Fix two base points $(0,0)$ and $(m,-n)$ with $m,n > 0$.   From these base points define  coupled LPP processes  $G^{(u)}_{(0,0),\,\bbullet}$ and $G^{(u)}_{(m,-n),\,\bbullet}$  whose boundary weights come from the increments of an  LPP process  $G_{u, \hspace{0.5pt}\bbullet}$ whose base point $u$ satisfies $u\le (0,0)$ and $u\le (m,-n)$.   Then for $z\in ((0,0)+\Z_{>0}^2)\cap((m,-n)+\Z_{>0}^2)$,  
$\exittime^{\,0\,\rightarrow \,z} \leq m$ if and only if $\exittime^{(m,-n)\,\rightarrow \,z} < -n$.
\end{lemma}

\begin{proof}[Proof of Lemma \ref{lemmaab}]
Let $a=\frac{1}{10}(1-\rho)$, $b=10\frac{2}{\rho^2}$.

It suffices to show that  if $r>0$ and  $N \geq N_0$ are such that \eqref{lambda} holds, then
\begin{align}
&\mathbb{P}^\lambda \big( \exittime^{\,0\,\rightarrow\, w_N} <arN^{2/3} \big) \leq e^{-Cr^3}, \label{noteasyone}\\
&\mathbb{P}^\lambda \big(\exittime^{\,0\,\rightarrow\, w_N} > brN^{2/3} \big) \leq e^{-Cr^3}.\label{easyone}
\end{align}

By  \eqref{cbound}   the distance between the origin and the black dot on the $x$-axis on the right of Figure \ref{fig99}   is bounded above by $\frac{2}{\rho^2}rN^{2/3} = \frac{1}{10}brN^{2/3}$. So the distance between the black dot and  $brN^{2/3} e_1$ is at least $brN^{2/3} - \frac{1}{10}brN^{2/3} = \frac{9}{10}brN^{2/3}$. 
Apply Lemma \ref{sec3lem2} to switch from the geodesic based at the origin to one based at the black dot, and apply  Theorem \ref{t:exit1} to the LPP process $G^{(0),\rho}_{\text{blackdot}, \,\bbullet}$:
$$\begin{aligned}\mathbb{P}^\lambda \big(\exittime^{\,0\,\rightarrow\, w_N} > brN^{2/3} \big) \leq \mathbb{P}^\lambda \big(\exittime^{\,0\,\rightarrow\, v_N} > brN^{2/3} \big)\\
\leq \mathbb{P}^\lambda \big(\exittime^{\text{ black dot }\rightarrow  \, v_N} \geq \tfrac{9}{10}brN^{2/3}\big) \leq e^{-Cr^3}.
\end{aligned}$$

To prove \eqref{noteasyone}   choose $h$  so  that   $(\floor{arN^{2/3}}, -h)$ is the  closest integer point to the $(-\xi[\lambda])$-directed ray starting at $w_N$  (see Figure \ref{figlemmaab}). Lemma \ref{relatetau} gives 
$$
\mathbb{P}^\lambda \big( \exittime^{\,0\,\rightarrow\, w_N} \leq \floor{arN^{2/3}} \big)  = \mathbb{P}^\lambda \big( \exittime^{\,(\floor{arN^{2/3}}, -h)\,\rightarrow\, w_N} < -h \big).
$$

\begin{figure}[t]
\captionsetup{width=0.8\textwidth}
\begin{center}
\begin{tikzpicture}[scale = 0.8]

\draw[lightgray, line width=0.3mm, ->] (-2-8,-2) -- (-2-8,4);
\draw[lightgray,line width=0.3mm, ->] (-2-8,-2) -- (3-8,-2);

\draw[gray, line width=0.3mm, ->] (0-8,0) -- (5-8,0);
\draw[gray, line width=0.3mm, ->] (0-8,0) -- (0-8,4);

\draw[ line width=0.3mm, ->] (1-8,-1) -- (4-8,-1);
\draw[ line width=0.3mm, ->] (1-8,-1) -- (1-8,2);

\draw[gray ,dotted, line width=1.1mm] (0-8,0) -- (0.7-8,0) -- (0.7-8,0.5) --(1-8,0.5) ;

\draw[lightgray ,dotted, line width=1.4mm] (1-8,0.5) --(2-8,0.5)-- (2-8,1) -- (4-8, 1) -- (4-8,3) ;
\draw[black ,dotted, line width=0.6mm] (1-8,0)  --(1-8,0.5) --(2-8,0.5)-- (2-8,1) -- (4-8, 1) -- (4-8,3) ;

\draw[black ,dotted, line width=1.1mm] (1-8,-1) -- (1-8,0.5);

\draw[lightgray, dotted, line width=0.7mm] (-2-8,-2)  --(-2-8,-0.8) --(0.7-8,-0.8) --(0.7-8,0) ;

\draw[ fill=lightgray](0-8,0)circle(1mm);
\node at (-0.3-8,-0.4) {$(0,0)$};

\draw[ fill=lightgray](1-8,-1)circle(1mm);
\node at (0.7-8,-1.4) {$(m,-n)$};

\draw[ fill=lightgray](-8-2,-2)circle(1mm);
\node at (-8-2.3,-2.3) {$u$};

\fill[color=white] (-8+4,3)circle(1.7mm); 
\draw[ fill=lightgray](-8+4,3)circle(1mm);
\node at (-8+4.3,3.3) {$z$};

--------------------------------------------------------------

\draw[gray ,dotted, line width=0.3mm](1,-1)--(4,3);

\draw[gray, line width=0.3mm, ->] (0,0) -- (6,0);
\draw[gray, line width=0.3mm, ->] (0,0) -- (0,4);

\draw[ line width=0.3mm, ->] (1,-1) -- (5,-1);
\draw[ line width=0.3mm, ->] (1,-1) -- (1,2);

\draw[ fill=black](2.2, 0)circle(1.3mm);

\draw[ fill=white](1.75, 0)circle(1.3mm);

\draw[gray ,dotted, line width=1.1mm] (0,0) -- (0.7,0) -- (0.7,0.5) --(1,0.5) ;

\draw[lightgray ,dotted, line width=1.4mm] (1,0.5) --(2,0.5)-- (2,1) -- (4, 1) -- (4,3) ;
\draw[black ,dotted, line width=0.6mm] (1,0)  --(1,0.5) --(2,0.5)-- (2,1) -- (4, 1) -- (4,3) ;

\draw[black ,dotted, line width=1.1mm] (1,-1) -- (1,0.5);

\draw[ fill=lightgray](0,0)circle(1mm);
\node at (-0.3,-0.4) {$(0,0)$};

\draw[ fill=lightgray](1,-1)circle(1mm);
\node at (0.7,-1.4) {$({arN^{2/3}},-h)$};

\fill[color=white] (4,3)circle(1.7mm); 
\draw[ fill=lightgray](4,3)circle(1mm);
\node at (4.3,3.3) {$w_N$};

\draw[ line width=1mm] (3,-0.1) -- (3,0.1);
\node at (3.3, -0.4) {${brN^{2/3}}$};

\end{tikzpicture}
\end{center}
\caption{\small {\it Left:} An illustration of Lemma \ref{relatetau}. As shown in the picture $\exittime^{(0,0)\rightarrow  \, z} \leq m$ if and only if $\exittime^{(m,-n)\rightarrow  \, z} < -n$. {\it Right:} Applying Lemma \ref{relatetau} in the proof of Lemma \ref{lemmaab} to assert that  $\mathbb{P}^\lambda \big( \exittime^{\,0\,\rightarrow\, w_N} \leq \floor{arN^{2/3}} \big)  = \mathbb{P}^\lambda \big( \exittime^{\,(\floor{arN^{2/3}}, -h)\,\rightarrow\, w_N} < -h \big).$}
\label{figlemmaab}
\end{figure}
Theorem \ref{t:exit1} states that it is unlikely for the $\lambda$-geodesic from $(\floor{arN^{2/3}}, -h)$ to $w_N$ to exit  late in the scale $N^{2/3}$  from the $y$-axis, because   the direction is the  characteristic one $\xi[\lambda]$. Thus it suffices to show $h$ is bounded below by some $k(\rho)rN^{2/3}$. 

Using the lower bound from $\eqref{cbound}$, the distance between the black dot and $(0,0)$ is bounded below by $(1-\rho)rN^{2/3} = 10arN^{2/3}$. The distance between the black dot and $\floor{arN^{2/3}}e_1$ is bounded below by $9arN^{2/3}$, and  the distance between the white dot and $\floor{arN^{2/3}}e_1$ is bounded below by $8arN^{2/3}$. The slope of the line going through $w_N$ and white dot is $\frac{\lambda^2}{(1-\lambda)^2}$. Thus, we have 
$$ h \geq \frac{\lambda^2}{(1-\lambda)^2}8arN^{2/3}.$$
Since $\lambda$ is bounded above and below by constants that depend on $\rho$, we get 
$$ h \geq k(\rho) rN^{2/3}$$
which finishes the proof.
\end{proof}

\begin{proof}[Proof of Theorem \ref{t:large-ub}]

For two fixed constants $0< a <b$, 
we increase  the weights on the intervals $\lzb \floor{arN^{2/3}}e_1,  \floor{brN^{2/3}}e_1 \rzb $ and $\lzb  \floor{arN^{2/3}}e_2,  \floor{brN^{2/3}}e_2\rzb$.  The new weights are chosen so  that their  characteristic directions  obey the  left diagram  of Figure \ref{fig21}  for  large $N \geq N_0(\rho)$. 

On the $e_1$-axis, define 
$$
\lambda = \rho + \frac{r}{N^{1/3}}
$$
and use $\Exp(1-\lambda)$ as the heavier weights.
The assumption  $ 0<r \leq [(1-\rho)^2\wedge\rho^2] N^{1/3}$ guarantees that $\rho<\lambda \leq \rho + (1-\rho)^2 < 1$.
On the $e_2$-axis, we define 
$$
\eta = \rho - \frac{r}{N^{1/3}},
$$
and the heavier weights are $\Exp(\eta)$. The condition $0< r \leq [(1-\rho)^2\wedge\rho^2] N^{1/3}$ guarantees that $0<\rho -(1-\rho)^2\wedge\rho^2 \leq \eta< \rho$.

Note that Lemma \ref{lemmaab} continues to hold if   $a$ is decreased and   $b$ is increased.   The constants   $a,b, N_0$ can always be adjusted so that  the situation in the  left diagram  of Figure \ref{fig21} appears. Later, we may decrease the value of $a$ further in our proof.

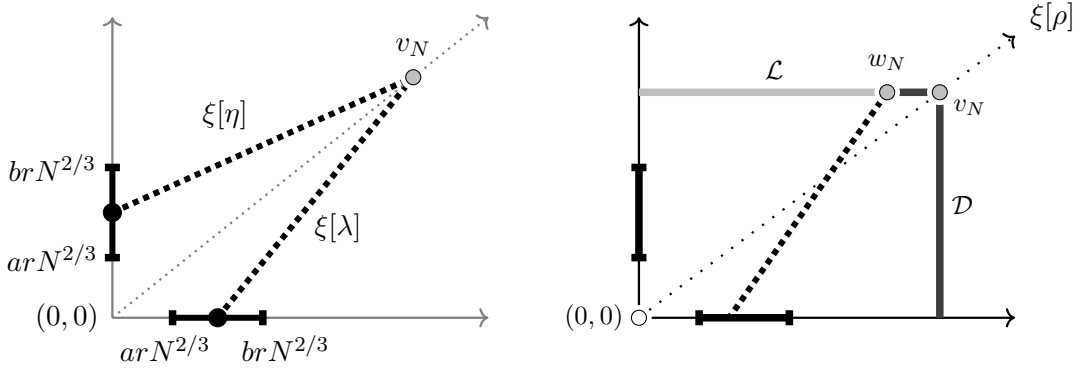
\begin{figure}[t]
\captionsetup{width=0.9\textwidth}
\begin{center}
\begin{tikzpicture}

\draw[gray, line width=0.3mm, ->] (0,0) -- (5,0);
\draw[gray, line width=0.3mm, ->] (0,0) -- (0,4);

\draw[ line width=0.8mm] (0.8,0) -- (2,0);
\draw[ line width=0.8mm] (0,0.8) -- (0,2);

\draw[ line width=1mm] (-0.1,0.8) -- (0.1,0.8);
\draw[ line width=1mm] (-0.1,2) -- (0.1,2);

\draw[ line width=1mm] (0.8,-0.1) -- (0.8,0.1);
\draw[ line width=1mm] (2,-0.1) -- (2,0.1);

\draw[dotted, line width=0.8mm] (1.4,0) -- (0.8*5, 0.8*4);
\draw[dotted, line width=0.8mm] (0,1.4) -- (0.8*5, 0.8*4);

\draw[gray ,dotted, line width=0.3mm, ->] (0,0) -- (5,4) ;
\fill[color=white] (0.8*5, 0.8*4)circle(1.7mm); 
\draw[ fill=lightgray](0.8*5, 0.8*4)circle(1mm);
\node at (0.8*5, 0.8*4+0.4) {$v_N$};

\node at (0.7, -0.4) {$arN^{2/3}$};
\node at (2.3, -0.4) {$brN^{2/3}$};

\node at (-0.8, 0.8) {$arN^{2/3}$};
\node at (-0.8, 2.0) {$brN^{2/3}$};

\node at (-0.6, 0) {$(0,0)$};

\draw[ fill=black](1.4, 0)circle(1.2mm);
\draw[ fill=black](0, 1.4)circle(1.2mm);

----------------------------------------
\draw[line width = 0.3mm, ->] (0+7,0)--(0+7,4);
\draw[line width = 0.3mm, ->] (0+7,0)--(5+7,0);

\draw[line width = 1mm, ] (0.8+7,0)--(2+7,0);
\draw[line width = 1mm, ] (0.8+7,-0.1)--(0.8+7,0.1);
\draw[line width = 1mm, ] (2+7,-0.1)--(2+7,0.1);

\draw[line width = 1mm, ] (0+7,0.8)--(0+7,2);
\draw[line width = 1mm, ] (-0.1+7,0.8)--(0.1+7,0.8);
\draw[line width = 1mm, ] (-0.1+7,2)--(0.1+7,2);

\draw[dotted, line width = 0.8mm, ] (3.3+7,3)--(1.2+7,0);

\draw[color=lightgray, line width = 1mm] (0+7,3) -- (3.3+7, 3);
\draw[ color = darkgray, line width = 1mm] (3.3+7,3) -- (4+7, 3) -- (4+7,0);

\draw[loosely dotted, line width = 0.3mm, ->] (0+7,0) -- (5+7, 3+3/4);

  (2.5+7,3.1) -- (4+7,3.1);

\fill[color=white] (3.3+7,3)circle(1.7mm);
\draw[ fill=lightgray](3.3+7,3)circle(1mm);
\node at (3.3+7,3.4) {\small $w_N$};
\node at (1.8+7,3.3) {\small $\cL$};
\fill[color=white] (0+7,0)circle(1.7mm);
\draw[ fill=white](0+7,0)circle(1mm);
\node at (-0.6+7,0) {\small $(0,0)$};

\fill[color=white] (4+7,3)circle(1.7mm);
\draw[ fill=lightgray](4+7,3)circle(1mm);
\node at (4.4+7,2.8) {\small $v_N$};
\node at (4.3+7,1.5) {\small $\cD$};

\node at (5.5+7, 4) {$\xi[\rho]$};

\node at (1.5, 2.7) {$\xi[\eta]$};

\node at (3, 1.2) {$\xi[\lambda]$};

\end{tikzpicture}
\end{center}
\caption{\small {\it Left:} Two dotted lines have slopes $\xi[\lambda]$ and $\xi[\eta]$.  {\it right:}  Decomposition of the north and east boundaries of $\lzb 0, v_N\rzb $ into   regions   $\cL$ (light gray)  and  $\cD$ (dark gray). A small perturbation of $v_N$ to $w_N$ keeps the endpoint of the $-\xi[\lambda]$ ray from $w_N$ in the interval $[arN^{2/3}, brN^{2/3}]$.}
\label{fig21}
\end{figure}

Recall the old environment of the stationary $\rho$-LPP process whose distribution is denoted by $\mathbb{P}^{\rho}$:
\begin{alignat*}{3}
\omega_z & \sim \Exp(1)  &&\text{for } z & &\in \Z^2_{>0}\\
\omega_{ke_1} & \sim \Exp(1-\rho)  \quad    &&\text{for } k & &\geq 1\\
\omega_{le_2} & \sim \Exp(\rho)   &&\text{for } l & &\geq 1.
\end{alignat*}
The new environment  $\wt\w$ increases the weights in the two intervals on the axes:
\begin{alignat*}{3}
\wt\omega_z & =\omega_z  &&\text{for } z & &\notin \lzb \floor{arN^{2/3}}e_1, \floor{brN^{2/3}}e_1 \rzb \cup  \lzb  \floor{arN^{2/3}}e_2,  \floor{brN^{2/3}}e_2\rzb \\
\wt\omega_{ke_1} & = \frac{1-\rho}{1-\lambda}\,\omega_{ke_1}  \quad  &&\text{for }ke_1 & &\in\lzb \floor{arN^{2/3}}e_1, \floor{brN^{2/3}}e_1\rzb \\
\wt\omega_{le_2} & = \frac{\rho}{\eta}\,\omega_{le_2}  &&\text{for }le_2 & &\in \lzb \floor{arN^{2/3}}e_2, \floor{brN^{2/3}}e_2\rzb.
\end{alignat*}
Denote the  probability measure for the  environment  $\wt\w$   by  $\wt{\mathbb{P}}$. 

The goal is  the  estimate
\beq\label{ezest}
\wt{\mathbb{P}}(A) \equiv \wt{\mathbb{P}} \big\{ \forall z \text{ outside } \lzb  0, v_N \rzb  \text{ we have } |\exittime^{\,0\,\rightarrow \,z}|  \geq \floor{ar N^{2/3}}\big\}  \geq 1/2
\eeq 
where $A$ denotes the event in braces.  
We check  that  this implies Theorem \ref{t:large-ub}. The Cauchy-Schwartz inequality gives
\beq  \label{rnest} 1/2 \leq \wt{\mathbb{P}}(A) = \mathbb{E}^\rho [\ind_A f] \leq \bigl(\mathbb{P}^\rho(A)\bigr)^{1/2} \bigl( \mathbb{E}^\rho[f^2]\bigr)^{1/2}\eeq
where $f=d\wt{\mathbb{P}}/d\P^\rho$ is the Radon-Nikodym derivative. Lemma \ref{l:radnik} gives the  bound  
\beq \label{rnest2}\mathbb{E}^\rho[f^2] \leq e^{Cr^3}\eeq
and then \eqref{rnest} and  \eqref{rnest2} imply the lower bound 
$$\mathbb{P}^\rho(A) \geq \tfrac14e^{-Cr^3}.   
$$ 
To replace the lower bound $\floor{arN^{2/3}}$ in   the event $A$ in \eqref{ezest}  with  $rN^{2/3}$ required for Theorem \ref{t:large-ub},  modify the constant $C$.

To show \eqref{ezest} we  bound its complement: 
\beq\label{ezest2}
\wt{\mathbb{P}} \big\{ \exists z \text{ outside } \lzb  0, v_N \rzb  \text{ such that } |\exittime^{\,0\,\rightarrow \,z}| <\floor{ ar N^{2/3}}\big\}  \le Cr^{-3}. 
\eeq 
We treat  the case $1\leq \exittime^{\,0\,\rightarrow \,z} < \floor{ar N^{2/3}}$  of \eqref{ezest2}.  The same arguments give the analogous bound for the case  $-\floor{arN^{2/3}} < \exittime \leq -1$. 
Start by perturbing the endpoint $v_N$ to a  new point $w_N$ as was done in Lemma \ref{lemmaab}: 
$$w_N = v_N - \floor{\tfrac1{10}(1-\rho)rN^{2/3}}e_1.$$
Break up the northeast boundary of $\lzb  0, v_N\rzb $ into two regions $\cL$ and $\cD$ as in the diagram on the  right of Figure \ref{fig21}. Note that the  $(-\xi[\lambda])$-directed ray started from $w_N$ still goes through    the interval $[arN^{2/3},brN^{2/3}]$. We now require $0< a < \tfrac1{10}(1-\rho)<10\frac{2}{\rho^2} <b$ for $a, b$ in order to apply Lemma \ref{lemmaab} directly in the later part of the proof.

First consider geodesics that hit $\cD$. 
We will show
\be\begin{aligned}\label{thm4.4dark}
    &\wt{\mathbb{P}} \big\{\exists z\in\cD: 1\leq \exittime^{\,0\,\rightarrow \,z} < \floor{ar N^{2/3}}\big\} 
    \leq Cr^{-3}.
\end{aligned}\ee
Let $\sigma_1^{\,0\,\to\, x}$ denote the exit time of the optimal path among those $0\,\to\,x$ paths whose first step is $e_1$. Then we have 
\be\begin{aligned}\label{thm4.4darkre}
    \wt{\mathbb{P}} \big\{\exists z\in\cD: 1\leq \exittime^{\,0\,\rightarrow \,z} < \floor{ar N^{2/3}}\big\}  
    &\leq \wt{\mathbb{P}} \big\{\exists z\in\cD:  \sigma_1^{0\,\rightarrow \,z}  < \floor{ar N^{2/3}}\big\}
    \\
    &\leq \wt{\mathbb{P}} \big\{\sigma_1^{0\,\rightarrow\,w_N}  < \floor{ar N^{2/3}}\big\} .
\end{aligned}\ee
The second inequality comes from the uniqueness of maximizing paths: the maximizing path  to $w_N$ cannot go to the right of a maximizing path  to $\cD$.

The task is to bound $\wt{\mathbb{P}} \big\{\sigma_1^{0\,\rightarrow\,w_N}  < \floor{ar N^{2/3}}\big\}$. Define an environment with  $\mathbb{P}^\lambda$ distribution  by multiplying the   $\mathbb{P}^\rho$  boundary weights  by   $\frac{1-\rho}{1-\lambda}$  on the $e_1$-axis and by   $\frac{\rho}{\lambda}$ on the  $e_2$-axis.   We have now three coupled weight configurations with marginal distributions  $\wt{\mathbb{P}}, \mathbb{P}^\rho$ and $\mathbb{P}^\lambda$.  Denote their joint distribution   by $\mathbb{P}$. Let   $\wt G$, $G^\rho$, and $G^\lambda$  denote the last-passage values under  these three environments. Additionally, let   $\wt G_{0, w_N}(I)$  denote the last-passage value restricted to  paths that  exit through the set $I$.

To obtain 
$$\wt{\mathbb{P}} \big\{\sigma_1^{0\,\rightarrow\,w_N}  < \floor{ar N^{2/3}}\big\} \leq Cr^{-3}$$
we show
\beq \label{welowerbound}\mathbb{P} \bigl\{\wt G_{0, w_N}(\lzb e_1, \floor{arN^{2/3}-1}e_1\rzb) < \wt G_{0, w_N}(\lzb \floor{arN^{2/3}}e_1, \floor{brN^{2/3}}e_1\rzb)\bigr\}\geq 1-Cr^{-3}.\eeq
By  Lemma \ref{lemmaab} there exists an event $A_1$ with $\mathbb{P}(A_1)\geq 1-e^{-Cr^3}$ such that on this event the geodesic of $G^\lambda_{0, w_N}$ exits inside $\lzb \floor{arN^{2/3}}e_1, \floor{brN^{2/3}}e_1\rzb$.
The following equality holds on $A_1$: 
$$\wt G_{0, w_N}(\lzb \floor{arN^{2/3}}e_1, \floor{brN^{2/3}}e_1\rzb) + \sum_{k=1}^{\floor{arN^{2/3}-1}} \bigg(\frac{1-\rho}{1-\lambda} - 1\bigg) \omega_{ke_1}  =  G^\lambda_{0, w_N}.$$
Together with the fact that
$$\wt G_{0, w_N}(\lzb e_1, \floor{arN^{2/3}-1}e_1\rzb) \leq G^\rho_{0, w_N},$$
the probability in \eqref{welowerbound} can be lower bounded as 
\beq \label{eq1}
\eqref{welowerbound} \geq \mathbb{P} \biggl(\biggl\{G^\rho_{0, w_N} < G^\lambda_{0, w_N} - \sum_{k=1}^{\floor{arN^{2/3}-1}} \bigg(\frac{1-\rho}{1-\lambda} - 1\bigg) \omega_{ke_1} \biggr\}\cap A_1\biggr).
\eeq
Up to a $\rho$-dependent constant  
\beq \label{ldbdry}\mathbb{E}\biggl[\;\sum_{k=1}^{\floor{arN^{2/3}-1}} \bigg(\frac{1-\rho}{1-\lambda} - 1\bigg) \omega_{ke_1} \biggr] \sim a r^2 N^{1/3},\eeq
and recall that  the parameter $a$ can be fixed arbitrarily small. On the other hand, a computation in eqn.~(5.53)   in the arXiv version of \cite{CGMlecture} with $\kappa_N^1 = -\floor{\tfrac1{10}(1-\rho)rN^{2/3}}$ and $\kappa_N^2 = 0$  gives 
\beq \label{diffexp} \mathbb{E}[G^\lambda_{0, w_N}] - \mathbb{E}[ G^\rho_{0, w_N}] \geq c_1r^2N^{1/3}
\eeq where $c_1$ is another  $\rho$-dependent constant.  Hence for small $a>0$ the event inside the  braces in \eqref{eq1}  should occur with high probability.  This we now demonstrate.
 
 Let 
 $$A_2=\{ G^\lambda_{0, w_N} > \mathbb{E}[ G^\rho_{0, w_N}] + \tfrac12{c_1}r^2N^{1/3}\}.$$
We  show that $\mathbb{P}(A_2) \geq 1-Cr^{-3}$.  First we estimate the variance $\Var[G^\rho_{0, w_N}] $.    The first equality below is Theorem 5.6 in the arXiv version of \cite{CGMlecture}:  
\begin{align}
\begin{split}\label{boundvar}
\Var[G^\rho_{0, w_N}] &= -\frac{\floor{(1-\rho)^2 N} - \floor{\tfrac1{10}(1-\rho)rN^{2/3}}}{(1-\rho)^2} + \frac{\floor{\rho^2 N} }{\rho^2} + \frac{2}{1-\rho}\mathbb{E}\bigg[\sum_{k=1}^{0\,\vee\,\exittime^{\,0\,\rightarrow \,w_N}} \w^\rho_{ke_1}\bigg]\\
& \leq CrN^{2/3} + \frac{2}{1-\rho}\mathbb{E}\bigg[\sum_{k=1}^{0\,\vee\,\exittime^{\,0\,\rightarrow \,v_N}} \w^\rho_{ke_1}\bigg] 
\leq CrN^{2/3} + C' N^{2/3}.
\end{split}
\end{align}
Shifting the endpoint from $w_N$ back to $v_N$ inside the expectations increases the expected value because $\exittime^{\,0\,\rightarrow \,w_N} \leq \exittime^{\,0\,\rightarrow \,v_N}$ almost surely. This gives the inequality between the two expectations. The last expectation   is of order $N^{2/3}$ as shown through Lemma 5.8 and Proposition 5.9 in the arXiv version of \cite{CGMlecture}.  Now we can bound:  
\begin{align*}
\mathbb{P}(A_2^c)&=\mathbb{P}\bigl(G^\lambda_{0, w_N} \leq \mathbb{E}[ G^\rho_{0, w_N}] + \frac{c_1}{2}r^2N^{1/3}\bigr) \\
(\text{using \eqref{diffexp}}) \qquad &\leq \mathbb{P}(G^\lambda_{0, w_N} \leq \mathbb{E}[ G^\lambda_{0, w_N}] - \frac{c_1}{2}r^2N^{1/3})\\&
\leq \frac{c_2}{r^4 N^{2/3}}\Var[G^\lambda_{0, w_N}]\\
(\text{Lemma 5.7, arXiv version of \cite{CGMlecture}})\qquad &\leq \frac{c_2}{r^4 N^{2/3}}(\Var[G^\rho_{0, w_N}] + c_3rN^{-1/3} (1-\rho)^2 N)
\leq C r^{-3}.
\end{align*}
For the last inequality we take $r\ge C'$ from the last line of \eqref{boundvar}. 
We have  the further  lower bound 
\beq\label{eq2} \eqref{eq1} \geq
\mathbb{P} \biggl(\biggl\{G^\rho_{0, w_N} < \mathbb{E}[ G^\rho_{0, w_N}] + \frac{c_1}{2}r^2N^{1/3} - \sum_{k=1}^{\floor{arN^{2/3}-1}} \bigg(\frac{1-\rho}{1-\lambda} - 1\bigg) \omega_{ke_1} \biggr\}\cap A_1 \cap A_2\biggr).\eeq
We handle the i.i.d.~sum above using large deviation of i.i.d.~exponential random variables. Let $I(\cdot)$ denote the Cram\'er rate function of the  $\Exp(1-\rho)$ distribution.  Then 
$$\mathbb{P}\bigg\{\left(\frac{1-\rho}{1-\lambda} - 1 \right) \sum_{k=1}^{\floor{arN^{2/3}-1}}  \omega_{ke_1} > \frac{c_1}{4} r^2 N^{1/3} \bigg\}  \le  e^{-arN^{2/3} I(c_5/a)} \leq e^{-c_6rN^{2/3}} $$ 
where $c_5$ is a certain constant, and for small enough $a>0$, $I(c_5/a) \ge c_6/a$.    Thus the event 
$$A_3=\biggl\{ \biggl(\frac{1-\rho}{1-\lambda} - 1 \biggr) \sum_{k=1}^{\floor{arN^{2/3}-1}}  \omega_{ke_1} \leq \frac{c_1}{4}r^2 N^{1/3}\biggr\} $$
satisfies $\mathbb{P}(A_3)\geq 1-e^{-c_6rN^{2/3}}$.  

Continuing the lower bound,
\beq \label{eq3}\eqref{eq2} \geq \mathbb{P} \left(\left\{G^\rho_{0, w_N} < \mathbb{E}[ G^\rho_{0, w_N}] + \frac{c_1}{4}r^2N^{1/3}  \right\}\cap A_1 \cap A_2 \cap A_3\right).\eeq
The variance bound from \eqref{boundvar} gives 
$$\mathbb{P}\left\{G^\rho_{0, w_N} - \mathbb{E}[ G^\rho_{0, w_N}] \geq  \frac{c_1}{4}r^2N^{1/3}  \right\} \leq \frac{c_2}{r^4 N^{2/3}}\Var[G^\rho_{0, w_N}] \leq Cr^{-3}.$$ 
All four events inside the probability   in  \eqref{eq3} have probability at least   $1-Cr^{-3}$.  We have verified  the estimate \eqref{welowerbound} and thereby completed the argument for the dark gray region $\mathcal{D}$.

Consider the light gray region $\cL$.  
The switch from $\wt{\mathbb{P}}$ to $\mathbb{P}^\rho$ decreases certain  boundary weights outside the range  $\lzb e_1, \floor{ar N^{2/3}-1}e_1\rzb$ and gives the first inequality below. 
\begin{align}\label{thm4.4light}
\begin{split}
    &\wt{\mathbb{P}} \big\{\exists z\in\cL: 1\leq \exittime^{\,0\,\rightarrow \,z}  < \floor{ar N^{2/3}}\big\}  \leq 
    {\mathbb{P}}^\rho \big\{\exists z\in\cL: 1\leq \exittime^{\,0\,\rightarrow \,z}  < \floor{ar N^{2/3}}\big\}\\[2pt] 
    &\qquad \leq {\mathbb{P}}^\rho \big\{\exists z\in\cL:   \exittime^{\,0\,\rightarrow \,z}  \ge 1 \big\}
    \leq {\mathbb{P}}^\rho \big\{  \exittime^{\,0\,\rightarrow\,w_N}\ge 1\big\}
    \leq e^{-Cr^3}. 
    \end{split}
\end{align}
  The last inequality follows from bound  \eqref{weprove2} in  Corollary \ref{sec3cor}.

Combining  \eqref{thm4.4dark} and  \eqref{thm4.4light} gives 
$$
\wt{\mathbb{P}} \big\{ \exists z \text{ outside } \lzb  0, v_N \rzb  \text{ such that } 1\leq \exittime^{\,0\,\rightarrow \,z} <\floor {ar N^{2/3}}\big\}  \le Cr^{-3}.
$$
The proof is complete. 
\end{proof}

The next theorem is the main intermediate result towards the  lower bound of Theorem \ref{t:main-small}.

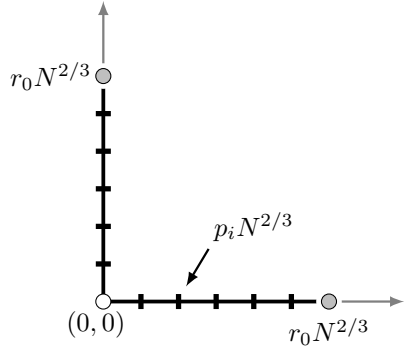
\begin{figure}[t]
\captionsetup{width=0.8\textwidth}
\begin{center}
 
\begin{tikzpicture}[>=latex, scale=1]

\draw[color = gray, line width = 0.3mm, ->] (0,0)--(0,4);
\draw[color = gray, line width = 0.3mm, ->] (0,0)--(4,0);

\draw[color = black, line width = 0.5mm] (0,0)--(0,3);
\draw[color = black, line width = 0.5mm] (0,0)--(3,0);

\draw[color = black, line width = 0.7mm] (0.5,-0.1)--(0.5,0.1);
\draw[color = black, line width = 0.7mm] (1,-0.1)--(1,0.1);
\draw[color = black, line width = 0.7mm] (1.5,-0.1)--(1.5,0.1);
\draw[color = black, line width = 0.7mm] (2,-0.1)--(2,0.1);
\draw[color = black, line width = 0.7mm] (2.5,-0.1)--(2.5,0.1);

\draw[color = black, line width = 0.7mm] (-0.1,0.5)--(0.1,0.5);
\draw[color = black, line width = 0.7mm] (-0.1,1)--(0.1,1);
\draw[color = black, line width = 0.7mm] (-0.1,1.5)--(0.1,1.5);
\draw[color = black, line width = 0.7mm] (-0.1,2)--(0.1,2);
\draw[color = black, line width = 0.7mm] (-0.1,2.5)--(0.1,2.5);

\node at (2,1) {\small $p_iN^{2/3}$};
\draw[color = black, line width = 0.3mm, ->] (1.4,0.7)--(1.1,0.2);

\fill[color=white] (3,0)circle(1.7mm); 
\draw[ fill=lightgray](3,0)circle(1mm);
\node at (3,-0.4) {\small $r_0N^{2/3}$};

\fill[color=white] (0,3)circle(1.7mm); 
\draw[ fill=lightgray](0,3)circle(1mm);
\node at (-0.7,3) {\small $r_0N^{2/3}$};

\draw[ fill=white ](0,0)circle(1mm);
\node at (-0.1,-0.3) {\small $(0,0)$};

\end{tikzpicture}
 
\end{center}
\caption{\small  Partition of the range of  $\exittime^{\,0\,\rightarrow\,v_N}$ in the event in \eqref{exit47}.   The origin is not necessarily a partition point.}
\label{fig13}
\end{figure}
\begin{theorem} \label{t:small-lb}   For each  $0<\rho<1$ 
there exist finite positive constants $\delta_0(\rho)$, $C(\rho)$ and  $N_0(\rho)$  such that 
for all  $N\geq N_0(\rho)$   and $ N^{-2/3}\leq\delta\le \delta_0(\rho)$, 
\[ \mathbb{P}^\rho \big\{ \exists z \text{ outside } \lzb  0, v_N \rzb  \text{ such that } |\exittime^{\,0\,\rightarrow \,z}| \leq \delta N^{2/3}   \big\}\geq C(\rho) \delta. \]
\end{theorem}

\begin{proof}  

Utilizing  Theorem \ref{t:exit1}, fix constants $r_0$, $C_0$ and  $N_0$ (depending on $\rho$)  such that, for $N\geq N_0$,  
\beq \label{exit47} \mathbb{P}^\rho \bigl\{ |\exittime^{\,0\,\rightarrow\,v_N+e_1+e_2} | \leq r_0 N^{2/3}\bigr\}\geq1/2.\eeq
Set   $v_N'=v_N+e_1+e_2 $.  
Given small $\delta> N^{-2/3}$,  partition $[-r_0,r_0]$ as 
\[    -r_0=p_0   < p_1 < \dotsm  <  p_{\floor{\frac{2r_0}{\delta}}}< p_{\floor{\frac{2r_0}{\delta}}+ 1} = r_0  \]
with mesh $p_{i+1}-p_i\le \delta$. 
See Figure \ref{fig13}.
By $(\ref{exit47})$ 
 there exists an integer $i^\star\in[ 0, \floor{\frac{2r_0}{\delta}}]$ such that 
\beq \label{istar}\mathbb{P}^\rho \big\{   p_{i^\star} N^{2/3} \le \exittime^{\,0\,\rightarrow\,v_N'} \le  p_{i^\star+1}N^{2/3} \big\} \geq \frac{\tfrac{1}{2} \delta}{2r_0} =  C(\rho) \delta.\eeq

We cannot control the exact location of $i^\star$. We compensate by varying the endpoint around $v_N'$. 
Let \[  A_N=\lzb v_N' -r_0 N^{2/3}e_1, v_N'\rzb \cup \lzb v_N'-r_0 N^{2/3}e_2, v_N'\rzb \]
denote   the set of lattice points on the boundary  of the rectangle $\lzb  0, v_N'\rzb $ within distance  $r_0 N^{2/3}$ of the upper right corner $v_N'$.   We claim  that for any integer $i\in [0,  \floor{\frac{2r_0}{\delta}}]$, 
\begin{align}  
\label{alli}
\mathbb{P}^\rho \bigl\{\exists z \in A_N   : |\exittime^{\,0\,\rightarrow\,z}|\leq \delta N^{2/3}\bigr\} 
 \geq \mathbb{P}^\rho \bigl\{p_{i} N^{2/3} \le \exittime^{\,0\,\rightarrow\,v_N'} \le  p_{i+1}N^{2/3}  \bigr\}.
\end{align}
Then bounds  \eqref{istar} and  \eqref{alli} imply
\beq\label{newest}\mathbb{P}^\rho \bigl\{\exists z \in A_N   : |\exittime^{\,0\,\rightarrow\,z}| \leq \delta N^{2/3}\bigr\} \geq C(\rho)\delta,\eeq
and Theorem \ref{t:small-lb}  directly follows from \eqref{newest}.

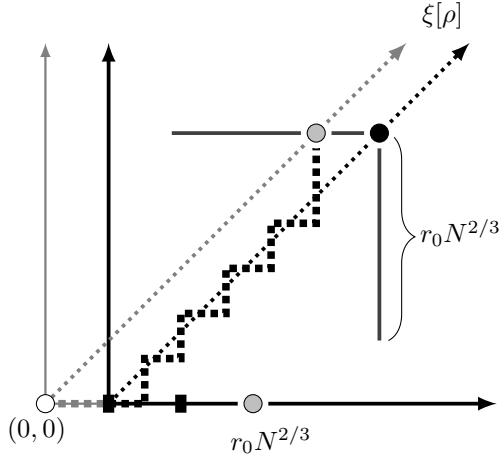
\begin{figure}[t]
\captionsetup{width=0.8\textwidth}
\begin{center}
 
\begin{tikzpicture}[>=latex, scale=1.2]

\draw[color = gray, line width = 0.3mm, ->] (0,0)--(0,4);
\draw[color = gray, line width = 0.3mm] (0,0)-- (0.7,0);
\draw[color = black, line width = 0.5mm, ->] (0.7,0)-- (5,0);

\draw[line width = 0.5mm,  ->] (0.7,0)--(0.7,4);

\draw[color=gray, dotted, line width = 1mm] (0,0) -- (0.7,0);

\draw[color = darkgray,  line width = 0.5mm,] (3.7-2.3,3)--(3.7,3) -- (3.7,3 -2.3);

\draw[color=black, dotted, line width = 1mm]  (0.7,0)--(1.1, 0)--(1.1,0.5) -- (1.5,0.5) -- (1.5,1)--(2,1)--(2,1.5)--(2.5,1.5) -- (2.5,2) -- (3,2) -- (3,3) ;
\draw [decorate,decoration={brace,amplitude=10pt}, xshift=0pt,yshift=0pt]
(3.8,3) -- (3.8,0.7) ;
\node at (4.6,1.9) {\small $r_0 N^{2/3} $};

\draw[color=gray, dotted, line width = 0.5mm, ->]  (0,0) -- (4,4);

\draw[color=black , dotted, line width = 0.5mm, ->]  (0.7,0) -- (4.7,4);

\node at (4.4,4.3 ) {\small $\xi[\rho]$};

\draw[line width = 1.5mm] (0.7,-0.1)--(0.7,0.1);
\draw[line width = 1.5mm] (1.5,-0.1)--(1.5,0.1);

\fill[color=white] (2.3,0)circle(1.7mm); 
\draw[ fill=lightgray](2.3,0)circle(1mm);
\node at (2.5,-0.4) {\small $r_0N^{2/3}$};

\fill[color=white] (3,3)circle(1.7mm); 
\draw[ fill=lightgray](3,3)circle(1mm);

\fill[color=white] (3.7,3)circle(1.7mm); 
\draw[ fill=black](3.7,3)circle(1mm);

\draw[ fill=white ](0,0)circle(1mm);
\node at (-0.1,-0.3) {\small $(0,0)$};

\end{tikzpicture}
 
\end{center}
\caption{\small The setup for proving $(\ref{alli})$.}
\label{fig14}
\end{figure}

 
We prove claim \eqref{alli}. If $p_i\le 0 \le  p_{i+1}$,  $(\ref{alli})$ is immediate.   We argue the case $p_{i+1} > p_i > 0$, the other one being analogous.  
 Set  $z=(\floor{p_i N^{2/3}} - 1)e_1$ and 
 apply Lemma \ref{sec3lem2} to the LPP process  $G^{(0),\rho}_{z, \,\bbullet}$.  Then 
  \begin{align}
\nn \mathbb{P}^\rho \bigl\{p_{i} N^{2/3} \le \exittime^{\,0\,\rightarrow\,v_N'} \le  p_{i+1}N^{2/3}  \bigr\} 
 \; &\le  \; \mathbb{P}^\rho \bigl\{1 \leq \exittime^{\,0 \,\rightarrow\,v_N'-(\floor{p_i N^{2/3}} - 1)e_1} \leq \delta N^{2/3}\bigr\} \\
\nn & \leq \; \mathbb{P}^\rho \bigl\{\exists z \in A_N   : |\exittime^{\,0\,\rightarrow\,z}| \leq \delta N^{2/3}\bigr\}.
\end{align}
\end{proof}

The remainder of this section proves the main intermediate result towards the upper bound of Theorem \ref{t:main-small}.  
It quantifies the lower bound on the exit point on the scale $N^{2/3}$. This  strengthens the estimates accessible without integrable probability, for previously no quantification was  attained (Theorem 2.2(b) in \cite{cuberoot}). The proof is based on the ideas from the recent work of \cite {balzs2019nonexistence, balzs2020local}.

\begin{theorem} \label{t:small-ub} For each  $0<\rho<1$ 
there exist finite positive constants $\delta_0(\rho)$, $C(\rho)$ and  $N_0(\rho)$  such that 
for all $ 0<\delta\le \delta_0(\rho)$  and $N\geq N_0(\rho)$, 
\[  \mathbb{P}^\rho \Big\{ \exists z \text{ outside } \lzb  0, v_N \rzb  \text{ such that } |\exittime^{\,0\,\rightarrow \,z}|\leq \delta N^{2/3}   \Big\} \leq C|\log\delta\hspace{0.5pt}|^{2/3} \delta.
\]
\end{theorem}

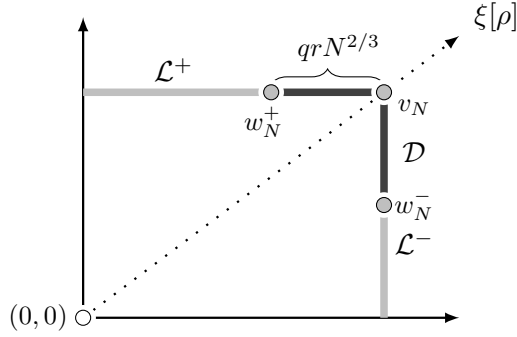
\begin{figure}[t]
\captionsetup{width=0.8\textwidth}
 \begin{center}
\begin{tikzpicture}[>=latex, scale=1]
\draw[line width = 0.3mm, ->] (0,0)--(0,4);
\draw[line width = 0.3mm, ->] (0,0)--(5,0);

\draw[color=lightgray, line width = 1mm] (0,3) -- (2.5, 3);
\draw[ color = darkgray, line width = 1mm] (2.5,3) -- (4, 3) -- (4,1.5);
\draw[color=lightgray, line width = 1mm] (4,1.5)--(4,0);

\draw[loosely dotted, line width = 0.3mm, ->] (0,0) -- (5, 3+3/4);

\draw [decorate,decoration={brace,amplitude=6pt,raise=0ex}]
  (2.5,3.1) -- (4,3.1);

\node at (3.4,3.6) {\small ${qrN^{2/3}}$};

\fill[color=white] (4,1.5)circle(1.7mm);
\draw[ fill=lightgray](4,1.5)circle(1mm);
\node at (4.4,1.5) {\small $w_N^-$};

\fill[color=white] (2.5,3)circle(1.7mm);
\draw[ fill=lightgray](2.5,3)circle(1mm);
\node at (2.4,2.6) {\small $w_N^+$};
\fill[color=white] (0,0)circle(1.7mm);
\draw[ fill=white](0,0)circle(1mm);
\node at (-0.6,0) {\small $(0,0)$};

\fill[color=white] (4,3)circle(1.7mm);
\draw[ fill=lightgray](4,3)circle(1mm);
\node at (4.4,2.8) {\small $v_N$};

\node at (5.5, 4) {$\xi[\rho]$};

\node at (4.4, 2.2){$\cD$};
\node at (1.2, 3.3){$\mathcal{L}^+$};

\node at (4.4, 1){$\mathcal{L}^-$};

\end{tikzpicture}

\end{center}
\caption{\small The north and east boundaries of $\lzb 0, v_N\rzb $ are decomposed into  $\cL^\pm$ (light gray) and $\cD$ (dark gray). The parameter $q$ is less than some small constant that depends only on $\rho$.}
\label{fig9}
\end{figure}

\begin{proof} 
We  prove the case $1\leq \exittime \leq \delta N^{2/3}$.  The proof for $-\delta N^{2/3}\leq \exittime \leq -1$ is similar. It suffices to look at the north and  east boundaries of $\lzb 0, v_N\rzb $  since any geodesic from $0$ to   outside of $\lzb 0, v_N\rzb $   crosses the boundary.   Decompose these boundaries into three parts $\cD$ and $\cL^\pm$  as in Figure \ref{fig9}, with
 $$ w_N^+ = v_N - \floor{qrN^{2/3}}e_1 \quad\text{and}\quad 
   w_N^- = v_N - \floor{qrN^{2/3}}e_2$$
where $q$ is a small positive constant chosen later, and $r = \big(|\log\delta\hspace{0.5pt}|/C)^{1/3}$ where $C$ is the constant  in the right-hand  side of the estimate in Theorem \ref{t:exit1}. The dark gray set $\cD$ comprises the vertices  between  $w_N^+$ and $w_N^-$ in the north-east corner of the boundary of the  rectangle $\lzb0, v_N \rzb$.

Consider first, the dark gray portion $\cD$.  Take 
  $0<\delta \leq \delta_0 = \tfrac{9}{10}$, where the bound   $\tfrac{9}{10}$ may be decreased later  in the proof. Our goal is to estimate 
\beq \mathbb{P}^\rho \{\exists z \in \text{$\cD$ such that } 1\leq \exittime^{\,0\,\to\, z}\leq \delta N^{2/3}\}. \label{goal8}
\eeq 
To do this, we place  the stationary LPP process on $0 + \mathbb{Z}^2_{\geq 0}$ as a nested  LPP process   inside a larger stationary LPP process on the quadrant  $-\floor{rN^{2/3}}e_1 + \mathbb{Z}^2_{\geq 0}$, as shown in Figure \ref{nest1}.
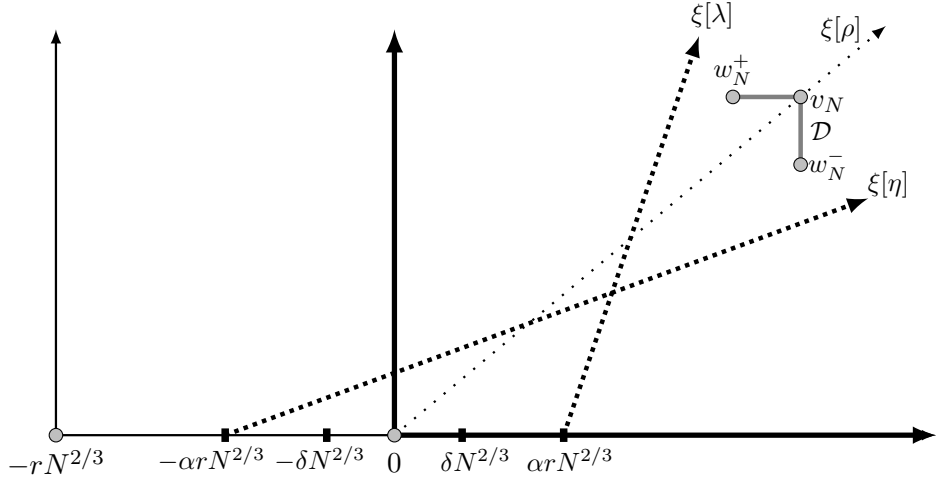
\begin{figure}[t]
\captionsetup{width=0.8\textwidth}
\begin{center}
 
\begin{tikzpicture}[>=latex, scale=0.9]


\draw[ line width = 0.6mm, ->, dotted] (-2.5,0)--(7,3.5);
\draw[ line width = 0.6mm, ->,  dotted] (2.5,0)--(4.5,5.9);
\node at (7.3,3.7) {\small $\xi[\eta]$};
\node at (4.7,6.2) {\small $\xi[\lambda]$};

\draw[ line width = 0.3mm, ->, loosely dotted] (0,0)--(6.6*1.1,5.5*1.1);
\node at (6.6,6) {\small $\xi[\rho]$};

\node at (6.3, 4.5) {\small $\cD$};
\draw[ line width = 0.6mm, color=gray] (5,5)--(6,5)--(6,4);

\draw[ fill=lightgray](5,5)circle(1mm);
\node at (5,5.4) {\small $w_N^+$};

\draw[ fill=lightgray](6,5)circle(1mm);
\node at (6.4,4.9) {$v_N$};

\draw[ fill=lightgray](6,4)circle(1mm);
\node at (6.4,4) {\small $w_N^-$};

\draw[line width = 0.3mm, ->] (-5,0)--(8,0);
\draw[line width = 0.3mm, ->] (-5,0)--(-5,6);

\draw[line width = 0.7mm, ->] (0,0)--(0,6);
\draw[line width = 0.7mm, ->] (0,0)--(8,0);

\draw[line width = 1mm] (1,-0.1) -- (1,0.1);
\node at (1.2, -0.35) {\small $\delta N^{2/3}$};

\draw[line width = 1mm] (2.5,-0.1) -- (2.5,0.1);
\node at (2.6, -0.35) {\small $\alpha rN^{2/3}$};

\draw[line width = 1mm] (-1,-0.1) -- (-1,0.1);
\node at (-1.1, -0.35) {\small $-\delta N^{2/3}$};

\draw[line width = 1mm] (-2.5,-0.1) -- (-2.5,0.1);
\node at (-2.7, -0.35) {\small $-\alpha rN^{2/3}$};

\draw[ fill=lightgray](-5,0)circle(1mm);
\node at (-5,-0.4) {$-rN^{2/3}$};

\draw[ fill=lightgray](0,0)circle(1mm);
\node at (-0,-0.4) {$0$};

\end{tikzpicture}
 
\end{center}
\caption{\small Illustration of the set $\mathcal {D}$, the nested LPP processes, and three characteristic directions. The parameters $q=\alpha$ are less than some small constant that depends only on $\rho$, $\delta$ is a small positive constant in $(0, \delta_0)$, and $r$ is a large constant with $r = ({|\log\delta\hspace{0.5pt}|}/{C})^{1/3}$.}\label{nest1}
\end{figure}
From the relation between geodesics of two nested LPP processes given in Lemma \ref{sec3lem2}, 
\begin{align*}
&\mathbb{P}^\rho \{\exists z \in \text{$\cD$  : }1\leq \exittime^{ \, 0\,\rightarrow\, z} \leq \delta N^{2/3}\,\}  \\
 &\quad \quad \leq \mathbb{P}^\rho \{\exists z \in \cD  : \floor{r N^{2/3}}-\delta N^{2/3}\leq \exittime^{ \, -\floor{r N^{2/3}}e_1\,\rightarrow\, z} \leq \floor{r N^{2/3}}+\delta N^{2/3}\,\}
\end{align*}
Thus, it suffices to obtain an upper bound for the second line above. To continue, we describe the rest of the setup shown in Figure \ref{nest1}.

The probability in \eqref{goal8}  vanishes  if $\delta N^{2/3}<1$ and hence we can always assume  
\be\label{Nass4}  N\ge \delta^{-3/2}.  \ee
Introduce the perturbed parameters
  \be\label{la8} 
  \lambda = \rho + \frac{r}{N^{1/3}} \quad \text{ and } \quad\eta= \rho - \frac{r}{N^{1/3}}.
  \ee
We require the following bounds to hold for these two parameters
\be\label{la9} \rho < \lambda \leq \rho + \frac{\rho\wedge(1-\rho)}{2} < 1 \quad \text{ and } \quad 0< \rho - \frac{\rho\wedge(1-\rho)}{2} \leq \eta < \rho. 
       \ee
The point of  the choice $\rho \pm \frac{\rho\wedge(1-\rho)}{2}$ is only to bound $\lambda$ and $\eta$ from above and below by two constants strictly inside $(0,1)$ and that  depend only on $\rho$. These two requirements can be rewritten as
$$
N\geq \left( \frac{2r}{\rho\wedge(1-\rho)}\right)^3.
$$
With \eqref{Nass4}, this bound on $N$  is automatically satisfied  as long as  $  \delta^{-3/2} \geq  \big( \frac{2r}{\rho\wedge(1-\rho)}\big)^3$.
With $r = \big(\tfrac{|\log\delta\hspace{0.5pt}|}{C}\big)^{1/3}$,  we can ensure this by considering $\delta>0$ subject to  
\be\label{dass8}   \delta \leq    \delta_0(\rho) = \left(\tfrac12{C(\rho\wedge(1-\rho))}\right)^3 \wedge \tfrac9{10}.\ee

Our next step is to fix $q$ and $\alpha$ small enough so that the $\xi[\eta]$- and $\xi[\lambda]$-directed rays started at the points $\pm \floor{\alpha rN^{2/3}} e_1$ avoid $\cD$ as shown in Figure \ref{nest1}. 
As in Figure \ref{fig99}, let  $u_N$ be the lattice point closest to where the $\xi[\lambda]$-ray from the origin crosses the north boundary of $[\![0, v_N ]\!]$.  
Then from  \eqref{cbound} we have 
$$  v_N\cdot e_1 - u_N \cdot e_1 \geq (1-\rho) rN^{2/3} .$$
Shift the starting point of the $\xi[\lambda]$-ray from the origin  to $\floor{\alpha r N^{2/3}}e_1$, and let $u_N'$ be  the new crossing point on the north boundary of $[\![0, v_N ]\!]$. By picking $q=\alpha=\frac{1-\rho }{10}$, the following lower bound holds:  
\beq \label{boundbelow} w_N^+\cdot e_1 - u_N' \cdot e_1 \geq \frac{1-\rho}{2} rN^{2/3}.\eeq
This gives us the desired picture for $\xi[\lambda]$ shown in Figure \ref{nest1}. The argument for the $\xi[\eta]$-directed ray is similar.  We may need to decrease $\alpha$ and $q$ further to achieve this but their values   depend only on $\rho$. 
At last, once $\alpha$ is fixed,   $r = \big(\tfrac{|\log\delta\hspace{0.5pt}|}{C}\big)^{1/3}$ allows us to decrease $\delta_0$ further so that  $\delta < \tfrac{1}{3}\alpha r$ for each $0< \delta \leq \delta_0$. This completes the description of  the setup   in Figure \ref{nest1}.

Now, to bound
$$
\mathbb{P}^\rho \big\{\exists z\in \cD:  \floor{r N^{2/3}}-\delta N^{2/3}\leq \exittime^{ \, - \floor{r N^{2/3}}e_1\,\rightarrow\, z} \leq  \floor{r N^{2/3}}+ \delta N^{2/3} \big\}, 
$$
we first bound the probability 
\beq\label{singleedge}
\mathbb{P}^\rho\big\{\exists z\in \cD:  \exittime^{ \, - \floor{r N^{2/3}}e_1\,\rightarrow\, z}  = \floor{r N^{2/3}} +t_0 \big\}  
\eeq
where $t_0$ is a fixed integer in $[\![-\floor{\delta N^{2/3}}, \floor{\delta N^{2/3}}]\!]$.

For $z\in \cD$ and $i\in [\![-\floor{\alpha rN^{2/3}}+1, \floor{\alpha rN^{2/3}}]\!]$, define horizontal increments  
$$
\widetilde{I}{}^z_i = G_{(i-1,1), z} - G_{(i,1), z}  
$$
  on the horizontal line $y=1$.  
Define a 2-sided walk $\{Z_n^{z, t_0}\}_{n\in[\![-\floor{\alpha rN^{2/3}}+1, \floor{\alpha rN^{2/3}}]\!]}$ by setting  $Z_{t_0}^{z, t_0} = 0$ and 
$$
Z_n^{z, t_0} - Z_{n-1}^{z, t_0} = I_n-\widetilde{I}{}^{z}_n .
$$
The boundary weights $I_n$ are those of the $\rho$-LPP process in the quadrant $-\floor{rN^{2/3}}e_1 + \mathbb{Z}^2_{\geq 0}$.  
On  the event 
$$\big\{\exittime^{ \, - \floor{r N^{2/3}}e_1\,\rightarrow\, z}  =  \floor{r N^{2/3}} +  t_0 \big\}$$  
the geodesic goes through the vertical unit edge $[\![(t_0,0), (t_0, 1) ]\!]$. This implies that 
the walk 
$\{Z_n^{z,t_0}\}_{n \in [\![-\floor{\alpha rN^{2/3}}+1, \floor{\alpha rN^{2/3}}]\!]}$ attains its unique  maximum at $n=t_0$.
To see this, 
note that for  $n\in [\![-\floor{\alpha rN^{2/3}}+1, \floor{\alpha rN^{2/3}}]\!] \setminus \{t_0\}$, we have almost surely
\begin{align}
G^\rho_{-\floor{rN^{2/3}}e_1, (t_0,0)} + G_{(t_0,1), z }  &>  G^\rho_{-\floor{rN^{2/3}}e_1, (n,0)} + G_{(n,1),z }\nonumber\\
\Longrightarrow\quad 
G^\rho_{-\floor{rN^{2/3}}e_1, (t_0,0)} - G^\rho_{-\floor{rN^{2/3}}e_1, (n,0)}   &>   G_{(n,1), z } - G_{(t_0,1), z }. \label{thisline}
\end{align}
From this, 
\begin{itemize}
\item 
 for   $n> t_0 $, \eqref{thisline}  $\implies$  $
-\sum_{i=t_0+1}^n I_i   >  -\sum_{i=t_0+1}^n \widetilde{I}{}^z_i$ $\implies$ 
$0 > Z_{n}^{z, t_0} - Z_{t_0}^{z, t_0} $; 

\item for $n<t_0 $, \eqref{thisline}  $\implies$  $
\sum_{i=n+1}^{t_0} I_i   >  \sum_{i=n+1}^{t_0} \widetilde{I}{}^z_i$  $\implies$  $Z_{t_0}^{z, t_0}  - Z_{n}^{z, t_0} > 0 $.

\end{itemize}

Since $\delta \leq \tfrac{1}{3}\alpha r$, $t_0 \in [-\tfrac{1}{3}\alpha r N^{2/3}, \tfrac{1}{3}\alpha r N^{2/3}]$. Also because the value of the walk at $t_0$ is zero, we now have 
\begin{align}
\eqref{singleedge} &\leq \mathbb{P} \Bigl\{ \exists z\in\cD:  \underset{n \in [\![-\floor{\alpha rN^{2/3}}+1, \floor{\alpha rN^{2/3}} ]\!]}{\text{argmax}} \{Z_n^{z,t_0}\} = t_0 \Bigr\} 
\nonumber\\
 & \leq \mathbb{P} \Bigl(  \big\{ \exists z\in \cD: Z_n^{z,t_0} < 0 \text{ for } n\in  \big(t_0, t_0+ \floor{\tfrac{1}{2}\alpha rN^{2/3}} \big] \big\}\label{withunion}\\
 &\quad \quad \quad\quad \quad \quad \quad \quad \quad  \bigcap  \big\{  \exists z\in \cD:  Z_n^{z,t_0} < 0 \text{ for } n\in   \big[t_0- \floor{\tfrac{1}{2}\alpha rN^{2/3}}, t_0 \big)\big\} \Bigr)\nonumber
\end{align}
Due to  the relative positions of $w_N^\pm$ and $z$,  Lemma \ref{sec3lem1} implies that 
\beq \label{I-34} 
\widetilde{I}{}^{\hspace{1.5pt}w_N^-}_i \leq \widetilde{I}{}^z_i \leq \widetilde{I}{}^{\hspace{1.5pt}w_N^+}_i
\quad\forall \, i\in [\![-\floor{\alpha rN^{2/3}}+1, \floor{\alpha rN^{2/3}}]\!] \ \text{ and }  \ z\in \cD. 
\eeq
Hence  for any $z\in \cD$,
\begin{align*}
Z_n^{z,t_0} \geq Z_n^{w_N^+,t_0} \quad \text{ for $n > t_0$}
\quad\text{and}\quad 
Z_n^{z,t_0} \geq Z_n^{w_N^-,t_0} \quad \text{ for $n < t_0$}.
\end{align*}
Therefore, we may bound \eqref{withunion} by
\begin{align}
\eqref{withunion} &\leq \mathbb{P} \bigg(\Big\{ Z_n^{w_N^+,t_0} < 0 \text{ for } n\in  \big(t_0, t_0+ \floor{\tfrac{1}{2}\alpha rN^{2/3}} \big] \Big\} \label{boundpm}\\
& \quad \quad \quad\quad \quad \quad \quad \quad \quad 
\bigcap  \Big\{ Z_n^{w_N^-,t_0} < 0 \text{ for } n\in   \big[t_0- \floor{\tfrac{1}{2}\alpha rN^{2/3}}, t_0 \big)\Big\}\bigg) .\nonumber
\end{align}

We bring the Busemann increments defined by the bulk weights $\{\omega_x\}_{x\in-\floor{rN^{2/3}}e_1 + \mathbb{Z}^2_{> 0}}$ into the picture.
\begin{figure}[t]
\captionsetup{width=0.8\textwidth}
 \begin{center}
 
\begin{tikzpicture}[>=latex, scale=0.7]

\draw[line width = 0.3mm, ->] (-5,0)--(8,0);
\draw[line width = 0.3mm, ->] (-5,0)--(-5,6);


\draw[line width = 0.3mm, ->] (6.7,5.7)--(-0,5.7);
\draw[line width = 0.3mm, ->] (6.7,5.7)--(6.7,2);



\node at (-1, 2.2) { $\widetilde{I}, I^\lambda$ and $I^\eta$ };

\draw[line width = 0.3mm, ->] (-1, 1.8) -- (-0.1, 0.9);

\draw[line width = 1mm, dotted] (-2.5,0.7)--(2.5,0.7);


\draw[line width = 1mm] (2.5,-0.1) -- (2.5,0.1);
\node at (2.6, -0.35) {\small $\alpha rN^{2/3}$};

\draw[line width = 0.3mm, ->] (1.9, -1.1) -- (1.1, -0.1);
\node at (2.1, -1.3) { $I$ };

\draw[line width = 1mm] (-2.5,-0.1) -- (-2.5,0.1);
\node at (-2.7, -0.35) {\small $-\alpha rN^{2/3}$};

\node at (3.5,6.1) {\small $B^\lambda$ and $B^\eta$};
\node at (8.1,3.5) {\small $B^\lambda$ and $B^\eta$};


\draw[ fill=lightgray](-5,0)circle(1mm);
\node at (-5,-0.5) {$-rN^{2/3}$};

\draw[ fill=lightgray](0,0)circle(1mm);
\node at (-0,-0.4) {$0$};

\draw[ fill=white](6.7,5.7)circle(1mm);
\node at (8.4,5.9) {$v_N + e_1+e_2$};

\draw[ line width = 0.6mm, color=gray] (5,5)--(6,5)--(6,4);

\draw[ fill=lightgray](5,5)circle(1mm);
\node at (4.5,5) {\small $w_N^+$};

\draw[ fill=lightgray](6,5)circle(1mm);

\draw[ fill=lightgray](6,4)circle(1mm);
\node at (6,3.6) {\small $w_N^-$};

\end{tikzpicture}
 
\end{center}
\caption{\small Setup for the stationary LPP processes with Busemann increments.}
\label{fig11aa}
\end{figure}
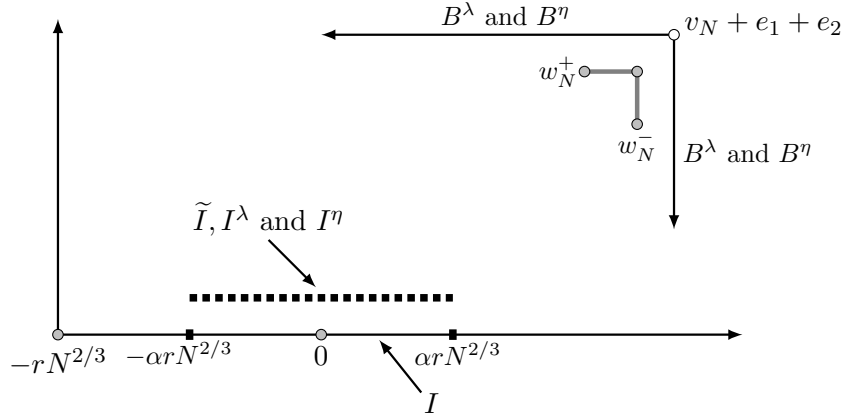
To each edge on the the north and east sides of the rectangle $[\![ -\floor{rN^{2/3}}e_1, v_N +e_1+e_2\rzb$, we attach   $\lambda$- and $\eta$-directed Busemann increments,   coupled as in Proposition \ref{indbuse}. This is depicted   in Figure \ref{fig11aa}. Together with the bulk weights in $[\![ -\floor{rN^{2/3}}e_1+e_2, v_N\rzb$, these define  stationary LPP processes with north and east boundaries, 
denoted by   $G^{\lambda, NE}_{x,v_N +e_1+e_2} $ and $G^{\eta, NE}_{x,v_N +e_1+e_2}$ for $x \in \lzb (-\floor{rN^{2/3}}, 1), v_N \rzb$. 
This is the  construction explained after Theorem \ref{t:buse}.   

On the horizontal line $y=1$  we have for $i\in [\![-\floor{\alpha rN^{2/3}}+1, \floor{\alpha rN^{2/3}}]\!]$  the increments 
\be\label{buseincrement}\begin{aligned}
I^\lambda_i &= G^{\lambda, NE}_{(i-1,1), v_N+ e_1 + e_2} - G^{\lambda, NE}_{(i,1), v_N+ e_1 + e_2}= B^\lambda_{(i-1,1), (i,1)} \\
\text{and} \qquad I^\eta_i &= G^{\eta, NE}_{(i-1,1), v_N+ e_1 + e_2} - G^{\eta, NE}_{(i,1), v_N+ e_1 + e_2} =  B^\eta_{(i-1,1), (i,1)}, 
\end{aligned}\ee
where the latter equalities are instances of \eqref{canbegeneral}. 


\begin{lemma}\label{lm:A78} 
The event 
\be\label{A78} 
A = 
\bigl\{ \forall \hspace{0.6pt} i\in [\![-\floor{\alpha rN^{2/3}}+1, \floor{\alpha rN^{2/3}}]\!]: I^\eta_i\leq \widetilde{I}{}^{\hspace{1.5pt}w_N^-}_i  \leq \widetilde{I}{}^{\hspace{1.5pt}w_N^+}_i \leq I^\lambda_i\bigr\} 
\ee
satisfies $\mathbb{P}(A^c) \leq e^{-Cr^3}$.
\end{lemma}
\begin{proof}

The middle inequality is already in \eqref{I-34}. 
We give the proof for
$$\mathbb{P}\bigl\{ \forall \hspace{0.6pt} i\in [\![-\floor{\alpha rN^{2/3}}+1, \floor{\alpha rN^{2/3}}]\!]:   \widetilde{I}{}^{\hspace{1.5pt}w_N^+}_i \leq I^\lambda_i \bigr\} 
\geq 1-e^{-Cr^3}.$$
The similar argument for the remaining part is omitted. 

We argue first  that $\widetilde{I}{}^{\hspace{1.5pt}w_N^+}_i \leq I^\lambda_i $ is implied for the entire range of indices  $i$ when the geodesic of $G^{\lambda, NE}_{(\floor{\alpha rN^{2/3}},1), v_N+e_1+e_2}$ exits the north boundary to the left of the point $w_N^++e_2$.

For $x\in\lzb (-\floor{rN^{2/3}},1), w_N^+ +e_2\rzb$,  let $G^{\lambda,N}_{x,w_N^++e_2}$ denote  the last-passage time from $x$ to $w_N^++e_2$ that  uses the $B^\lambda$ increment weights on the north boundary (superscript $N$ for north).     

The exit time $\exittime^{\lambda, NE, \,x\,\rightarrow\,v_N+e_1+e_2}$ records the signed distance from  the vertex  $v_N+e_1+e_2$ to the point where the geodesic of   $G^{\lambda, NE}_{x,v_N +e_1+e_2} $  enters the north (as a positive value) or the east (as a negative value) boundary of the rectangle $\lzb x, v_N +e_1+e_2\rzb$.  Since geodesics cannot cross, the event
$$\left\{\exittime^{\lambda, NE, \,(\floor{\alpha rN^{2/3}}, 1)\,\rightarrow\,v_N+e_1+e_2}> qrN^{2/3}\right\}$$ implies $$\bigcap_{i\in [\![-\floor{\alpha rN^{2/3}}+1, \floor{\alpha rN^{2/3}}]\!]}\left\{\exittime^{\lambda, NE, \,(i, 1)\,\rightarrow\,v_N+e_1+e_2}> qrN^{2/3}\right\}. $$
This further implies
\begin{align}
G^{\lambda,N}_{(i-1,1),w_N^++e_2} - &G^{\lambda,N}_{(i,1),w_N^+ +e_2} =G^{\lambda,NE}_{(i-1,1),v_N+e_1+e_2} -G^{\lambda,NE}_{(i,1),v_N+e_1+e_2}\label{onevent}\\[3pt]  
&\quad \quad \quad \quad \quad \quad \quad\quad \quad \quad \quad \quad\forall {i\in [\![-\floor{\alpha rN^{2/3}}+1, \floor{\alpha rN^{2/3}}]\!]}. \nonumber 
\end{align}
 In the derivation below,  Lemma $\ref{sec3lem1}$ gives the first  inequality.   The equality in the second line is \eqref{onevent} which is valid on 
the event $\bigl\{\exittime^{\lambda,NE, \,(\floor{\alpha rN^{2/3}}, 1)\,\rightarrow\,v_N+e_1+e_2}> qrN^{2/3}\bigr\}$: 
\begin{align*}
 \nonumber \widetilde{I}{}^{\hspace{1.5pt}w_N^+}_i = G_{(i-1,1), w_N^+} -  G_{(i,1), w_N^+} 
&\leq G^{\lambda,N}_{(i-1,1), w_N^++e_2} -  G^{\lambda,N}_{(i,1), w_N^++e_2}\\
& = G^{\lambda,NE}_{(i-1,1), v_N+e_1+e_2} -  G^{\lambda,NE}_{(i,1), v_N+e_1+e_2} 
= I^\lambda_i\\[3pt]   
&\quad \quad \quad \quad \quad \quad \quad\quad \quad \quad \quad \quad\forall {i\in [\![-\floor{\alpha rN^{2/3}}+1, \floor{\alpha rN^{2/3}}]\!]}. 
\end{align*}
This finishes the proof   that 
$\exittime^{\lambda, NE, (\floor{\alpha rN^{2/3}}, 1)\,\rightarrow\,v_N+e_1+e_2}> qrN^{2/3}$ implies  $ \widetilde{I}{}^{\hspace{1.5pt}w_N^+}_i \leq I^\lambda_i$ for all ${i\in [\![-\floor{\alpha rN^{2/3}}+1, \floor{\alpha rN^{2/3}}]\!]}$. 

Finally, we show that 
$$
\mathbb{P}\left\{\exittime^{\lambda, NE, \,(\floor{\alpha rN^{2/3}}, 1)\,\rightarrow\,v_N+e_1+e_2}> qrN^{2/3}\right\} \geq 1-e^{-Cr^3}.
$$
This follows from the standard exit time estimate. As shown in the left diagram of  Figure \ref{fig111}, the geodesic of  $G^{\lambda,NE}_{(\floor{\alpha rN^{2/3}}, 1), v_N+e_1+e_2}$ (gray dotted line) tends to follow the characteristic direction $\xi[\lambda]$ which means it enters the north boundary on the left of $w_N^+ + e_2$.
\begin{figure}[t]
\captionsetup{width=0.8\textwidth}
 \begin{center}
 
\begin{tikzpicture}[>=latex, scale=1.2]

\draw[line width = 0.3mm, ->, lightgray] (0,0)--(3,0);
\draw[line width = 0.3mm, ->, lightgray ] (0,0)--(0,2);

\draw[line width = 0.3mm, ->] (4.2,3)--(4.2-4,3);
\draw[line width = 0.3mm, ->] (4.2,3)--(4.2,3-3);

\draw[line width = 0.3mm, loosely dotted, ->] (1,0.5) -- (3,3.5);
\node at (3.1,3.9) {\small $\xi[\lambda]$};
\draw[dotted, color=gray, line width = 1mm] (1,0.5) -- (1,1) --(1.5, 1) -- (1.5,1.5) --(2,1.5) -- (2,2.5)--(2.5,2.5)--(2.5,3) -- (4.2,3);

\draw[ fill=lightgray](4.2,3)circle(1mm);
\node at (5.2,3) {\small $v_N^++e_1+e_2$};

\draw[ fill=lightgray](0,0)circle(1mm);
\node at (0,-0.4) {\small $(0,0)$};

\draw[ fill=lightgray](3.5,3)circle(1mm);
\node at (3.5,2.6) {\small $w_N^++e_2$};

\draw[ fill=lightgray](1,0.5)circle(1mm);
\node at (2,0.5) {\small $(\alpha rN^{2/3},1)$};


\draw[line width = 0.3mm, ->, gray] (7+4.2,3)--(7+4.2-4,3);
\draw[line width = 0.3mm, ->, gray] (7+4.2,3)--(7+4.2,3-3);

\draw[line width = 0.3mm, loosely dotted, ->] (7+1,0.5) -- (7+1+3,3.5+1.5);
\draw[dotted, color=lightgray, line width = 1.3mm] (7+1,0.5) -- (7+1,1) --(7+1.5, 1) -- (7+1.5,1.5) --(7+2,1.5) -- (7+2,2)--(7+3,2)--(7+3,2.5) -- (7+3.76,2.5)-- (7+3.76,3) -- (7+4.2,3)  ;
\draw[dotted, color=darkgray, line width = 1mm](7+3.5,4.26)--(10.5, 2.5);
\draw[dotted, color=black, line width = 0.5mm] (7+1,0.5) -- (7+1,1) --(7+1.5, 1) -- (7+1.5,1.5) --(7+2,1.5) -- (7+2,2)--(7+3,2)--(7+3,2.5) -- (7+3.5,2.5);

\draw[line width = 0.3mm, ->] (7+3.5,4.26)--(7,4.26);
\draw[line width = 0.3mm, ->] (7+3.5,4.26)--(7+3.5,1.3);

\draw[ fill=black](7+3.5,4.26)circle(1mm);

\node at (11,4.5) {\small $\xi[\lambda]$};

\draw[ fill=lightgray](7+4.2,3)circle(1mm);
\node at (7+5.2,3) {\small $v_N^++e_1+e_2$};

\draw[ fill=white](7+2.65,3)circle(1mm);

\draw[ fill=lightgray](7+3.5,3)circle(1mm);

\draw[ fill=lightgray](7+1,0.5)circle(1mm);
\node at (7+2,0.5) {\small $(\alpha rN^{2/3},1)$};

\draw[line width = 0.3mm, ->] (9.5, 3.6)--(10.25, 3.4);
\node at (9-0.1,3.65) {\small triangle};

\node at (1.3,3.3) {\small $B^\lambda$};
\node at (4.5,1) {\small $B^\lambda$};

\end{tikzpicture}
 
\end{center}
\caption{\small \textit{Left:} The likely behavior of the geodesic of $G^{\lambda,NE}_{(\floor{\alpha rN^{2/3}}, 1), v_N+e_1+e_2}$. It enters the north boundary to  the left of $w_N^+ +e_2$. \textit{Right:} The unlikely behavior of the geodesic of $G^{\lambda,NE}_{(\floor{\alpha rN^{2/3}}, 1), v_N+e_1+e_2}$. In this case, the dark dotted line is the  geodesic between the black dot and $(\floor{\alpha rN^{2/3}}, 1)$. It spends an atypically large amount of time on the boundary. }
\label{fig111}
\end{figure}
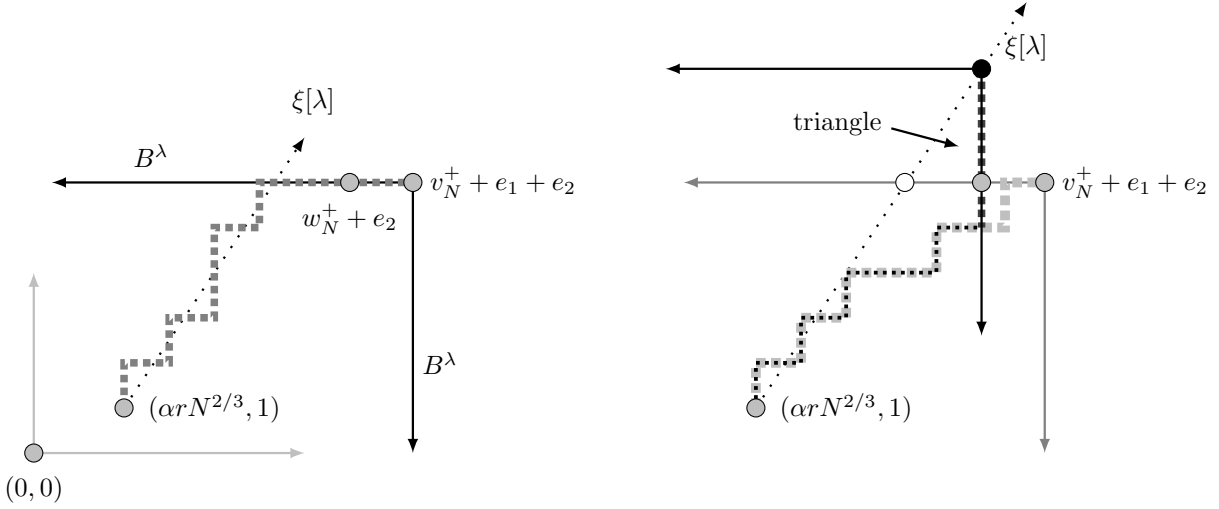
Else, by Lemma \ref{relatetau}, there  exists a parameter-$\lambda$ stationary LPP process whose  geodesic (black dotted line in the right diagram of Figure \ref{fig111}) in the characteristic direction spends excessive time on the boundary. The precise argument goes as follows.

Consider   the right triangle whose vertices are  the black, gray and white dots highlighted in the right diagram  of Figure $\ref{fig111}$. The distance between the white and gray dots is bounded below by $\frac{1-\rho}{2} rN^{2/3}$  by  \eqref{boundbelow}. Then, the distance between the black dot and the gray dot is at least 
$ \frac{\lambda^2}{(1-\lambda)^2}\frac{1-\rho}{2} rN^{2/3}$
where $\frac{\lambda^2}{(1-\lambda)^2}$ is the slope of the hypotenuse. 
By  Theorem \ref{t:exit1}, the probability that  the geodesic shown as the black dotted line remains  on the boundary throughout the segment between the black and the gray dot  is bounded above by $e^{-Cr^3}$.   Here $C$ depends on $\lambda$,  and bounds  \eqref{la9} turn this into a dependence  on $\rho$.   This completes the proof of Lemma \ref{lm:A78}. 
%
%
\end{proof}

With these new horizontal increments $I^\lambda$ and $I^\eta$, define two more 2-sided random walks $Z^{\lambda, t_0}_n$ and $Z^{\eta, t_0}_n$  with  $Z^{\lambda, t_0}_{t_0}= Z^{\eta, t_0}_{t_0} = 0$ and 
\begin{align*}
Z_n^{\lambda, t_0} - Z_{n-1}^{\lambda, t_0} &= I_n-{I}{}^{\lambda}_n,\\
Z_n^{\eta, t_0} - Z_{n-1}^{\eta, t_0} &= I_n-{I}{}^{\eta}_n,
\end{align*}
On the event $A$ from \eqref{A78}, 
\begin{align*}
Z_n^{\lambda, t_0} \leq Z_n^{w_N^+, t_0} \text{ for $n > t_0$} 
\quad\text{and}\quad 
Z_n^{\eta, t_0} \leq Z_n^{w_N^-, t_0}\text{ for $n < t_0$} .
\end{align*}
We continue our bound 
\begin{align}
\mathbb{P}(\text{event in }\eqref{boundpm} \cap A)   &\leq \mathbb{P} \bigg(\Big\{ Z_n^{\lambda,t_0} < 0 \text{ for } n\in  \big(t_0, t_0+ \floor{\tfrac{1}{2}\alpha rN^{2/3}} \big] \Big\}  \label{boundwithind}\\
& \quad \quad \quad\quad \quad \quad 
\bigcap \Big\{ Z_n^{\eta,t_0} < 0 \text{ for } n\in   \big[t_0- \floor{\tfrac{1}{2}\alpha rN^{2/3}}, t_0 \big)\Big\}\bigg). \nonumber
\end{align}
From Proposition \ref{indbuse}, the increment  variables $\{I^\lambda_{(i,1)}\}_{i> t_0} \cup \{I^{\eta}_{(i,1)}\}_{i \leq t_0}$ are independent, and these are independent of the boundary weights $\{I_i\}$ by construction. Thus, the two events on the right-hand  side above are independent. This gives
\begin{align*} 
\eqref{boundwithind} &= \mathbb{P}\Big\{ Z_n^{\lambda,t_0} < 0 \text{ for } n\in  \big(t_0, t_0+ \floor{\tfrac{1}{2}\alpha rN^{2/3}} \big] \Big\}\\
&\qquad\qquad\qquad\qquad\qquad
 \cdot \,\mathbb{P} \Big\{ Z_n^{\eta,t_0} < 0 \text{ for } n\in   \big[t_0- \floor{\tfrac{1}{2}\alpha rN^{2/3}}, t_0 \big)\Big\}. 
\end{align*} 
The steps of the random walks in the  two probabilities above have distributions $\Exp(1-\rho) - \Exp(1-\lambda)$ and $\Exp(1-\eta) - \Exp(1-\rho)$, respectively. By Lemma \ref{rw} each of the probabilities  is bounded above by $C(\rho) rN^{-1/3}$ where $C(\rho)$ is a constant that depends only on $\rho$ by virtue of \eqref{la9}. 

To summarize, we have shown
\begin{align*}
\mathbb{P}^\rho &\{\exists z \in \cD :  \exittime^{ \, -\floor{r N^{2/3}}e_1\,\rightarrow\, z}  =\floor{r N^{2/3}}+ t_0\,\} \\
& \leq  \mathbb{P}({A}^c) +\mathbb{P}^\rho \big(\{\exists z \in \cD :  \exittime^{ \, -\floor{r N^{2/3}}e_1\,\rightarrow\, z}  = \floor{r N^{2/3}}+t_0\,\} \cap A\big)\\
& \leq  e^{-Cr^3} + \big(C(\rho) rN^{-1/3})^2. 
\end{align*}
With a union bound over $t_0$, 
\begin{align*}
 \mathbb{P}^\rho &\{\exists z\in \cD: \floor{r N^{2/3}}-\delta N^{2/3}\leq \exittime^{ \, - \floor{r N^{2/3}}e_1\,\rightarrow\, z} \leq \floor{r N^{2/3}}+\delta N^{2/3} \} \\
& \leq   \mathbb{P}({A}^c) +\mathbb{P}^\rho \big( \{\exists z\in \cD: \floor{r N^{2/3}}-\delta N^{2/3}\leq \exittime^{ \, - \floor{r N^{2/3}}e_1\,\rightarrow\, z} \leq \floor{r N^{2/3}}+\delta N^{2/3} \} \cap A\big)\\
& \leq  e^{-Cr^3}  + (2\delta N^{2/3})\big(C(\rho) rN^{-1/3})^2 \\
& = e^{-Cr^3} + C(\rho)2 \delta r^2. 
\end{align*}
Letting $r = \left(C^{-1}{|\log \delta|}\right)^{1/3}$, this gives the desired upper bound $C(\rho) \delta |\log\delta\hspace{0.5pt}|^{2/3}$ with a new constant $C(\rho)$. This completes the proof for the dark region $\cD$ of Figure \ref{fig9}. 

For geodesics that enter  $\mathcal{L}^+$ we use monotonicity that comes from uniqueness of finite geodesics: 
\begin{align*}
&\mathbb{P}^\rho \big\{\exists v\in\cL^+ : 1\leq\exittime^{\,0\,\rightarrow \,v} \leq \delta N^{2/3} \big\} \leq \mathbb{P}^\rho \big\{\exists v\in\cL^+ :  \exittime^{\,0\,\rightarrow \,v}\ge 1  \big\}  \\
&\qquad \leq \mathbb{P}^\rho \big\{ \exittime^{\,0\,\rightarrow\,w_N^+} \ge 1 \big\}  
\leq e^{-Cr^3} = \delta.
\end{align*}
The last inequality comes from  bound \eqref{weprove2} from Corollary \ref{sec3cor}.

\begin{figure}[t]
\captionsetup{width=0.8\textwidth}
\begin{center}
\begin{tikzpicture}[scale = 1]

\draw[gray ,dotted, line width=0.3mm](1,-1)--(4,3);

\draw[gray, line width=0.3mm, ->] (0,0) -- (6,0);
\draw[gray, line width=0.3mm, ->] (0,0) -- (0,4);

\draw[ line width=0.3mm, ->] (1,-1) -- (5,-1);
\draw[ line width=0.3mm, ->] (1,-1) -- (1,2);

\draw[ fill=white](1.75, 0)circle(1.3mm);
\draw[ fill=black ](1,0)circle(1.3mm);

\draw[gray ,dotted, line width=1.1mm] (0,0) -- (0.7,0) -- (0.7,0.5) --(1,0.5) ;

\draw[lightgray ,dotted, line width=1.4mm] (1,0.5) --(2,0.5)-- (2,1) -- (4, 1) -- (4,3) ;
\draw[black ,dotted, line width=0.6mm] (1,0)  --(1,0.5) --(2,0.5)-- (2,1) -- (4, 1) -- (4,3) ;

\draw[black ,dotted, line width=1.1mm] (1,-1) -- (1,0.5);

\draw[ fill=lightgray](0,0)circle(1mm);
\node at (-0.3,-0.4) {$(0,0)$};

\draw[ fill=lightgray](1,-1)circle(1mm);
\node at (0.7,-1.4) {$({\delta N^{2/3}},-h)$};

\fill[color=white] (4,3)circle(1.7mm); 
\draw[ fill=lightgray](4,3)circle(1mm);
\node at (4.3,3.3) {$w_N^-$};

\end{tikzpicture}
\end{center}
\caption{\small From Lemma \ref{relatetau}, if $ \exittime^{\,0\,\rightarrow \,w_N^-}\le \delta N^{2/3}$ (gray dotted line), then $\exittime^{\,(\floor{\delta N^{2/3}}, -h)\,\rightarrow \,w_N^-}\le  -h$ (black dotted line).}
\label{lminus}
\end{figure}

For geodesics that enter $\mathcal{L}^-$, this follows from Lemma \ref{relatetau}. First, from the uniqueness of finite geodesics, it suffices to look at the point $w_N^-$ since 
\begin{align*}
\mathbb{P}^\rho \big\{\exists v\in\cL^- : 1\leq\exittime^{\,0\,\rightarrow \,v} \leq \delta N^{2/3} \big\} \leq \mathbb{P}^\rho \big\{ \exittime^{\,0\,\rightarrow \,w_N^-}\le \delta N^{2/3}  \big\}.
\end{align*}
Trace back a $(-\xi[\rho])$-directed ray from the point $w_N^-$.
Up to a $\rho$-dependent constant, this ray  crosses the $x$-axis  at 
$\floor{\frac{(1-\rho)^2}{\rho^2}qrN^{2/3}}e_1$  (the white dot in Figure \ref{lminus}). Decrease  $\delta_0$ further if necessary so that 
$\delta< \delta_0 \leq \frac{(1-\rho)^2}{2\rho^2}qr.$
Then the distance between the black and white dots in Figure \ref{lminus}  is at least  $\frac{(1-\rho)^2}{2\rho^2}qrN^{2/3}$.

Let $h$ be the positive integer such that $(\floor{\delta rN^{2/3}}, -h)$ is the closest lattice point to the $(-\xi[\rho])$-directed ray from  $w_N^-$. Then, $h\geq \frac{1}{2}qrN^{2/3}$.
From Lemma \ref{sec3lem1}, whenever $ \exittime^{\,0\,\rightarrow \,w_N^-}\le \delta N^{2/3}$ (gray dotted line), then  $\exittime^{\,(\floor{\delta N^{2/3}}, -h)\,\rightarrow \,w_N^-}< -h$ (black dotted line).   Theorem \ref{t:exit1}  bounds this probability  by $e^{-Cr^3}$. This  completes the proof of Theorem \ref{t:small-ub}.
\end{proof}

\section{Dual geodesics and proofs of the main theorems} \label{s:proofs}

The main theorems from Section \ref{s:main} are proved by applying the exit time bounds of Section \ref{s:exit-pf} to dual geodesics that live on the dual lattice.  First define south and west directed  semi-infinite paths (superscript sw) in terms of the Busemann functions from Theorem \ref{t:buse}: 
\beq\label{bg16} \begin{aligned}  
\mathbf{b}^{{\rm sw},\rho, x}_0&=x, \quad \text{and for $k\ge 0$}  \\[4pt] 
\mathbf{b}^{{\rm sw},\rho, x}_{k+1}&=\begin{cases}   \mathbf{b}^{{\rm sw},\rho, x}_{k} - e_1, &\text{if } \ B^\rho_{\mathbf{b}^{{\rm sw},\rho, x}_{k}-e_1,\,\mathbf{b}^{{\rm sw},\rho, x}_{k}}\le   B^\rho_{\mathbf{b}^{{\rm sw},\rho, x}_{k}-e_2,\,\mathbf{b}^{{\rm sw},\rho, x}_{k}}
\\[6pt]  
 \mathbf{b}^{{\rm sw},\rho, x}_{k} - e_2, &\text{if } \   B^\rho_{\mathbf{b}^{{\rm sw},\rho, x}_{k}-e_2,\,\mathbf{b}^{{\rm sw},\rho, x}_{k}} <   B^\rho_{\mathbf{b}^{{\rm sw},\rho, x}_{k}-e_1,\,\mathbf{b}^{{\rm sw},\rho, x}_{k}} . 

\end{cases} 
\end{aligned} \eeq
Recall  the dual weights  
$\{\widecheck{\omega}^\rho_x = B^\rho_{x-e_1,x}\wedge B^\rho_{x-e_2,x}\}_{x\in \Z^2}$ introduced in part (iii) of Theorem \ref{t:buse}. 


Let  $e^*=\tfrac12(e_1+e_2)=(\tfrac12,\tfrac12)$ denote the shift between the lattice $\Z^2$ and its dual $\Z^{2*}=\Z^2+e^*$.   Shift the dual weights to the dual lattice by defining 
$\w^{*}_z=\widecheck{\omega}^\rho_{z+e^*}$
 for $z\in\Z^{2*}$.  By Theorem \ref{t:buse}(iii) these weights are i.i.d.\ Exp(1).   The LPP process for these weights is defined as in \eqref{sec2G}: 
\beq\label{sec2G*}
G^{*}_{x,y} = \max_{z_{\bbullet}\, \in\, \Pi^{x,y}} \sum_{k=0}^{|y-x|_1} \omega^{*}_{z_k}.
\eeq
Shift the southwest paths to the dual lattice by defining 
\[  \mathbf{b}^{*,\rho,z}_{k} =  \mathbf{b}^{{\rm sw},\rho,z+e^*}_{k}-e^* \qquad\text{  for $z\in\Z^{2*}$ and $k\geq0$.} \] 
These definitions  reproduce on the dual lattice  the semi-infinite geodesic setting  described in  Section \ref{s:buse}, with reflected lattice directions.   This is captured in the next theorem that summarizes the development from Section 4.2 of \cite{coalnew}.

\begin{theorem} \label{t:dual}   Fix $\rho\in(0,1)$.   Then the following hold almost surely.  
\begin{enumerate}[{\rm(i)}] \itemsep=3pt 
\item  For each $z\in\Z^{2*}$, the path $\mathbf{b}^{*,\rho,z}$  is the unique $(-\xi[\rho])$-directed semi-infinite geodesic from $z$ in the LPP process \eqref{sec2G*}.
  Precisely, 
\[   \lim_{n\to\infty} \frac{\mathbf{b}^{*,\rho,z}_{n}}{n} = -\xi[\rho]
\quad \text{and}\quad 
 \forall k<l \text{ in } \Z_{\geq0}:  
 G^*_{\mathbf{b}^{*,\rho,z}_{l}, \mathbf{b}^{*,\rho,z}_{k}}= \sum_{i=k}^l \omega^{*}_{\mathbf{b}^{*,\rho,z}_{i}} .
 \]
\item The semi-infinite geodesics and the dual semi-infinite geodesics are equal in distribution, modulo the $e^*$-shift and lattice reflection: 
 $ \{\mathbf{b}^{*,\rho,z}\}_{z\in \Z^{*2}}\;\overset{d}=\;  \{ - e^* -\bgeod{\,\rho}{-(z+e^*)}\}_{z\in \Z^{*2}}$.  
\item The collections of paths 
 $ \{\bgeod{\,\rho}{z}\}_{z\in \Z^2}$ and  $ \{\mathbf{b}^{*,\rho,z}\}_{z\in \Z^{*2}}$ almost surely never cross each other.
\end{enumerate} 
\end{theorem} 



Part (ii), the distributional equality of the tree of directed geodesics and the dual, was first proved  in \cite{dual}.    The non-crossing property of part (iii) can be seen from a simple picture.  The additivity of the  Busemann functions gives 
\beq \label{addB}B^\rho_{x,x+e_1} + B^\rho_{x+e_1, x+e_1+e_2} = B^\rho_{x,x+e_2} + B^\rho_{x+e_2, x+e_1+e_2}.\eeq
By \eqref{busegeo}  $\bgeod{\,\rho}{x}_1=x+e_1$ if and only if 
$B^\rho_{x,x+e_1} \leq B^\rho_{x,x+e_2}$.  By \eqref{addB} this is equivalent to 
$B^\rho_{x+e_2, x+e_1+e_2}\leq  B^\rho_{x+e_1, x+e_1+e_2} $ which is the same as 
$\mathbf{b}^{{\rm sw},\rho, x+e_1+e_2}_1=x+e_2$, and  this last  is equivalent to 
$\mathbf{b}^{*,\rho,x+e^*}_{k} =x+e^*-e_1$.  An analogous argument works for the $e_2$ step.  The conclusion is that  the increments of  $\bgeod{\,\rho}{\hspace{0.5pt}\bbullet}$  out of $x$ and  $\mathbf{b}^{*,\rho,\hspace{0.5pt}\bbullet}$ out of $x+e^*$ cannot cross.  See Figure \ref{disjointness}.

\begin{figure}[t]
\captionsetup{width=0.8\textwidth}
\begin{center}
 
\begin{tikzpicture}[>=latex, scale=0.75]

\draw[gray, line width = 0.8mm, ->] (0,0)--(1.8,0);
\draw[gray,  dotted, line width = 0.8mm, ->] (0.8,0.8)--(-1,0.8);
\draw[lightgray, line width = 0.8mm, ->] (1.6,1.6)--(-0.2,1.6);

\draw[ fill=lightgray](0,0)circle(1mm);
\node at (-0.3,-0) {$x$};

\draw[ fill=lightgray](1.6,1.6)circle(1mm);
\node at (1.6+1.4,1.6) {$x+e_1+e_2$};

\draw[ fill=lightgray](0.8,0.8)circle(1mm);
\node at (1+0.8,0.8) {$x+e^*$};

\end{tikzpicture}
\caption{\small  The equivalent events $\bgeod{\,\rho}{x}_1=x+e_1$ (dark gray arrow),   
$\mathbf{b}^{{\rm sw},\rho, x+e_1+e_2}_1=x+e_2$ (light gray arrow), and $\mathbf{b}^{*,\rho,x+e^*}_{k} =x+e^*-e_1$ (dotted arrow).   The dark gray and dotted arrows never cross.  
} \label{disjointness}
\end{center}
\end{figure}
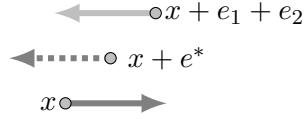

To connect the dual semi-infinite geodesics with $\rho$-geodesics, define a stationary LPP process $G^{*,\,\rho}_{-e^*, \hspace{0.5pt}\bbullet}$ exactly as in \eqref{sec2G^rho} with boundary weights on the south and east boundaries, but on the dual quadrant $-e^*+\Z_{\ge0}^2$ based at $-e^*$.  The boundary weights are defined by shifting Busemann function values to the dual lattice: 
\[   I^{*,\,\rho}_{-e^*+ke_1}= B^\rho_{(k-1)e_1, ke_1}
\quad\text{and}\quad 
J^{*,\,\rho}_{-e^*+le_2}= B^\rho_{(l-1)e_1, le_1}.  
\] 
The bulk weights are $\{\w^*_x: x\in\Z^{*2}, x\ge e^*\}$. 


\begin{proposition} \label{sec3prop1}
For any $w\in e^*+\Z_{\ge0}^2$ the following holds.  The edges  of the semi-infinite geodesic  $\mathbf{b}^{*,\rho, w}$ that have at least one endpoint in  $e^*+\Z_{\geq 0}^2$  are also edges of the  geodesic of  $G^{*,\,\rho}_{-e^*, w}$. 
\end{proposition}

Proposition \ref{sec3prop1}, illustrated in Figure \ref{fig5}, is another version of Lemma \ref{sec3lem2a}. It is proved as Prop.~5.1 in \cite{coalnew} but without the shift to the dual lattice, so in terms of the southwest geodesics in \eqref{bg16} for the weights $\widecheck{\omega}^\rho$.

\begin{figure}[t]
\captionsetup{width=0.8\textwidth}
\begin{center}
 
\begin{tikzpicture}[>=latex, scale=0.8]

\draw[color=lightgray,line width = 0.3mm, ->] (0,-1)--(0,3);
\draw[color=lightgray,line width = 0.3mm, ->] (0,-1)--(5,-1);

\draw[color=gray, dotted, line width = 1mm, <-] (-1.5, 0)-- (0.5,0) --(0.5,1)-- (2.5,1)--(2.5,1.5) -- (2.5,2) -- (3.5,2) ;

\fill[color=white] (0,-1)circle(1.7mm); 
\draw[ fill=white](0,-1)circle(1mm);
\node at (-0.6,0-1) {\small $-e^*$};

\fill[color=white] (3.5,2)circle(1.7mm); 
\draw[ fill=lightgray](3.5,2)circle(1mm);
\node at (3.8,2) {\small $w$};

\fill[color=white] (0,0)circle(1.7mm); 
\draw[ fill=lightgray](0,0)circle(1mm);

\draw[color=lightgray,line width = 0.3mm, ->] (0+7.5,-1)--(0+7.5,3);
\draw[color=lightgray,line width = 0.3mm, ->] (0+7.5,-1)--(5+7.5,-1);

\draw[color=darkgray, dotted, line width = 1.2mm,] (0+7.5,-1) --(0+7.5,0)-- (0.5+7.5,0) --(0.5+7.5,1)-- (2.5+7.5,1) --(2.5+7.5,1.5) -- (2.5+7.5,2) -- (3.5+7.5,2) ;

\draw[color=gray, dotted, line width = 1mm, <-] (-1.5+7.5,0) -- (0+7.5,0);
\fill[color=white] (0+7.5,-1)circle(1.7mm); 
\draw[ fill=white](0+7.5,-1)circle(1mm);
\node at (-0.6+7.5,0-1) {\small $-e^*$};

\fill[color=white] (3.5+7.5,2)circle(1.7mm); 
\draw[ fill=lightgray](3.5+7.5,2)circle(1mm);
\node at (3.8+7.5,2) {\small $w$};

\fill[color=white] (7.5,0)circle(1.7mm); 
\draw[ fill=lightgray](7.5,0)circle(1mm);

\end{tikzpicture}
 
\end{center}
\caption{\small Illustration of Proposition \ref{sec3prop1}.  On the left the dual semi-infinite geodesic $\mathbf{b}^{*,\rho, w}$ (light dotted path). On the right the  geodesic of  $G^{*,\,\rho}_{-e^*, w}$ (dark dotted path). The two paths coincide in the bulk.}
\label{fig5}
\end{figure}
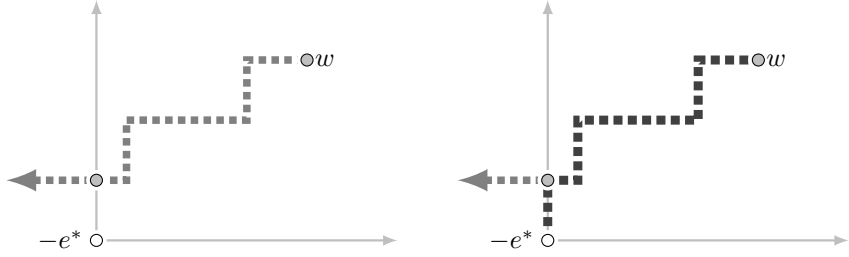

We are ready to prove the  main results.

\begin{proof}[Proof of Theorem \ref{t:main-small}]  
Referring to Figure \ref{fig12},  geodesics   $\bgeod{\,\rho}{(0, \fl{\delta N^{2/3}})}$ and $\bgeod{\,\rho}{(\fl{\delta N^{2/3}},0)}$  (gray dotted lines) coalesce outside $\lzb 0, v_N\rzb $ if and only if  some  dual geodesic started outside of $\lzb0,v_N\rzb - e^*$  (black dotted line) enters the square $\lzb (0,0), (\floor{\delta N^{2/3}},\floor{\delta N^{2/3}})\rzb $. From Proposition \ref{sec3prop1}, the restrictions of these dual geodesics are the $\rho$-geodesics of the stationary LPP process on $-e^* + \Z^2_{\geq 0}$ with Busemann boundary weights on the south and west. Consequently
\be\label{r804}\begin{aligned}
&\mathbb{P} \big\{\coal^\rho( \fl{\delta N^{2/3}}e_1 , \fl{\delta N^{2/3}}e_2) \not\in \lzb 0, v_N\rzb\big\} 
=\mathbb{P}^\rho\bigl\{ \exists z\notin \lzb  0, v_N \rzb  : \,  |\exittime^{\,0\,\rightarrow \,z} | \leq \delta N^{2/3}\bigr\} .
\end{aligned}\ee
The bounds claimed in Theorem \ref{t:main-small} follow from Theorems \ref{t:small-lb} and  \ref{t:small-ub}.
 \end{proof}

\begin{figure}[t]
\captionsetup{width=0.8\textwidth}
\begin{center}
 
\begin{tikzpicture}[>=latex, scale=1.3]

\draw[line width = 0.3mm, ->] (0,0)--(0,4);
\draw[line width = 0.3mm, ->] (0,0)--(5,0);

\draw[color=gray, dotted, line width = 1mm, ->] (1,0)--(1,0.5) -- (1.5,0.5) -- (1.5,1)--(2,1)--(2,1.5)--(2.5,1.5) -- (2.5,3) ;

\draw[color=gray, dotted, line width = 1mm, ->] (0,1)--(1,1) -- (1,1.8) -- (1.7,1.8)--(1.7,3.5) -- (2.6,3.5);

\draw[color=gray, dotted, line width = 0.5mm]  (3,0) -- (3,2)--(0,2);

\draw[ dotted, line width = 1mm] (2.1,2.2)--(2.1, 1.65) -- (1.6, 1.65) -- (1.6, 1.2) --(1.2,1.2) -- (1.2,0.7) -- (0.4,0.7) -- (0.4, 0.5) -- (-0.1, 0.5)--(-0.1,-0.1);

\draw[ dotted, line width = 0.6mm, ->]  (0,0.5) --(-0.7,0.5) --  (-0.7, -0.5) --(-1,-0.5) -- (-1,-1)  ;

\draw[ fill=black](2.1,2.2)circle(1mm);
\node at (2.1,2.5) {\small $x^*$};

\fill[color=white] (0,1)circle(1.7mm); 
\draw[ fill=lightgray](0,1)circle(1mm);
\node at (-0.9,1) {\small $(0, \delta N^{2/3})$};

\fill[color=white] (1,0)circle(1.7mm); 
\draw[ fill=lightgray](1,0)circle(1mm);
\node at (1,-0.4) {\small $(\delta N^{2/3},0)$};

\fill[color=white] (3,2)circle(1.7mm); 
\draw[ fill=lightgray](3,2)circle(1mm);
\node at (3.6,2) {\small $v_N-e^*$};

\draw[ fill=black](-0.1,-0.1)circle(1mm);
\node at (-0.1,-0.4) {\small $-e^*$};

\end{tikzpicture}
 
\end{center}
\caption{\small Geodesics $\bgeod{\,\rho}{(\floor{\delta N^{2/3}},0)}$ and  $\bgeod{\,\rho}{(0,\floor{\delta N^{2/3}})}$   (gray dotted lines)  coalesce outside $\lzb0,v_N\rzb$.  Equivalently, some dual point $x^*$ outside of $\lzb0,v_N\rzb - e^*$ sends a  dual geodesic (black dotted line)  into the rectangle $\lzb (0,0), (\floor{\delta N^{2/3}},\,\floor{\delta N^{2/3}})\rzb $.}
\label{fig12}
\end{figure}



\begin{proof}[Proof of Theorem \ref{t:main-large}] 
 Referring to   Figure \ref{fig20},  geodesics   $\bgeod{\,\rho}{(0, \fl{rN^{2/3}})}$ and $\bgeod{\,\rho}{(\fl{rN^{2/3}},0)}$  (gray dotted lines) coalesce inside $\lzb 0, v_N\rzb $ if and only if  every  dual geodesic started from the north and east boundaries of $\lzb-e^*,v_N + e^*\rzb$  (black dotted lines) avoids the square $\lzb (0,0), (\floor{rN^{2/3}}, \floor{rN^{2/3}})\rzb $.  From Proposition \ref{sec3prop1}, the restrictions of these dual geodesics are the $\rho$-geodesics of the stationary LPP process  on $-e^* + \Z^2_{\geq 0}$ with Busemann boundary weights on the south and west,   
\be\label{r800}  \begin{aligned}
&\mathbb{P}\bigl\{ \coal^\rho( \fl{rN^{2/3}}e_1, \fl{rN^{2/3}}e_2)\in \lzb 0, v_N\rzb \bigr\} 
=\mathbb{P}^\rho\bigl\{ \forall z\notin \lzb  0, v_N \rzb  : \,  |\exittime^{\,0\,\rightarrow \,z}|  \geq r N^{2/3}\bigr\} . 
\end{aligned}\ee 
The lower bound claimed in Theorem \ref{t:main-large} follows  from Theorem \ref{t:large-ub}.  The claimed upper bound is a trivial weakening of Theorem \ref{t:exit1}. 
\end{proof}


\begin{figure}[t]
\captionsetup{width=0.8\textwidth}
\begin{center}
\begin{tikzpicture}[>=latex, scale=1.3]
\draw[fill=lightgray ] (1,0) -- (1,1) --(0,1) -- (0,0)-- (1,0);

\draw[gray, line width=0.3mm, ->] (0,0) -- (5,0);
\draw[gray, line width=0.3mm, ->] (0,0) -- (0,4);

\draw[gray, dotted, line width=1mm, ->] (1,0) -- (1.5,0) -- (1.5,1) -- (2,1)-- (2,2) -- (3,2) -- (3,4);
\draw[gray, dotted, line width=1mm] (0,1) -- (0.4,1) -- (0.4,2)--  (1,2) -- (2,2);

\draw[black, dotted, line width=1mm] (2.8,3.2) -- (2.8,2.5) -- (1, 2.5) -- (1,2.3) -- (0.2, 2.3) -- (0.2,1.2) -- (0, 1.2);

\draw[black, dotted, line width= 0.6mm, ->] (0,1.2) -- (-1,1.2);

\draw[black, dotted, line width=1mm] (3.2,3.2) -- (3.2,0.8)  -- (1.9,0.8) -- (1.9,0);

\draw[black, dotted, line width=0.6 mm, ->] (1.9,0) -- (1.9,-1) ;

\draw[black, dotted, line width=1mm,] (1.9,-0.1) -- (-0.1,0-0.1)  -- (-0.1,1.2)  ;

\fill[color=white] (2.8,3.2)circle(1.7mm); 
\draw[ fill=lightgray](2.8,3.2)circle(1mm);

\fill[color=white] (3.2,3.2)circle(1.7mm); 
\draw[ fill=lightgray](3.2,3.2)circle(1mm);

\draw[gray ,dotted, line width=0.3mm, ] (0.8*5,0) -- (0.8*5,0.8*4) --(0,0.8*4) ;

\fill[color=white] (0.8*5, 0.8*4)circle(1.7mm); 
\draw[ fill=lightgray](0.8*5, 0.8*4)circle(1mm);
\node at (0.8*5 + 0.7, 0.8*4) {$v_N+e^*$};

\draw[black, line width=1mm] (1, -0.2) -- (1, 0.1);
\draw[black, line width=1mm] (-0.2, 1) -- (0.1, 1);
\node at (1, -0.5) {$(rN^{2/3}, 0)$};

\node at (-0.9, 0.8) {$(0, rN^{2/3})$};
\node at (-0.5, -0.2) {$-e^*$};

\fill[color=white] (2.8,3.2)circle(1.7mm); 
\draw[ fill=lightgray](2.8,3.2)circle(1mm);
\node at (2.7,3.5) {\small $u^*$};

\fill[color=white] (3.2,3.2)circle(1.7mm); 
\draw[ fill=lightgray](3.2,3.2)circle(1mm);
\node at (3.6,3.5) {\small $u^*+e_1$};

\draw[ fill=black](-0.1,-0.1)circle(1mm);

\end{tikzpicture}
\end{center}
\caption{\small None of the the $\rho$-geodesics will enter the gray square because they are bounded away by the two dual geodesics (black dotted lines) drawn above.  }
\label{fig20}
\end{figure}

\begin{proof}[Proof of Corollary \ref{samelevel}] From the duality, it suffices to show 
\begin{enumerate}[(i)]
\item $\mathbb{P}^\rho \big\{\exists z \text{ outside } \lzb 0, v_N\rzb  \text{ such that } 1\leq \exittime^{\,0\,\to\,z} \leq \delta N^{2/3}\big\} \geq C_1\delta;$
\item $\mathbb{P}^\rho \big\{\exists z \text{ outside } \lzb 0, v_N\rzb  \text{ such that } 1\leq \exittime^{\,0\,\to\,z} \leq r N^{2/3}\big\} \geq 1-e^{-C_2 r^3}.$
\end{enumerate}
We establish  (ii) from the special case 
\beq \label{samelevelii}\mathbb{P}^\rho \big\{1\leq \exittime^{\,0\,\to \,v_N+\floor{\frac{1}{10}rN^{2/3}}e_1} \leq r N^{2/3}\big\} \geq 1-e^{-C_2r^2}. \eeq

Furthermore, from  \eqref{samelevelii} the proof of Theorem \ref{t:small-lb} can be adapted to prove (i), by partitioning 
$[0, rN^{2/3}]$ into intervals  of size $\le\delta rN^{2/3}$ and repeating the argument.

Inequality  \eqref{samelevelii} comes from the estimates
\begin{align}
&\mathbb{P}^\rho \big\{\exittime^{\,0\,\to\,  v_N+\floor{\frac{1}{10}rN^{2/3}}e_1} \leq -1 \big\} \leq e^{-Cr^3} \label{samelevel1}\\
&\mathbb{P}^\rho \big\{\exittime^{\,0\,\to\,  v_N+\floor{\frac{1}{10}rN^{2/3}}e_1} > r N^{2/3}\big\} \leq e^{-Cr^3} \label{samelevel2}.
\end{align}

Inequality \eqref{samelevel1} is bound  \eqref{weprove} of  Corollary \ref{sec3cor}.    For \eqref{samelevel2},   apply Lemma \ref{sec3lem2} to the process $G^{(0),\,\rho}_{z, \,\bbullet}$ with the new base point $z=\floor{\frac{1}{10}rN^{2/3}}e_1$, and then   Theorem \ref{t:exit1}: 
\begin{align*}  
&\mathbb{P}^\rho \big\{\exittime^{\,0\,\to\, v_N+\floor{\frac{1}{10}rN^{2/3}}e_1} \ge  r N^{2/3}\big\}  
 \leq  \mathbb{P}^\rho \big\{\exittime^{\,0\,\to\, v_N} \ge   \tfrac9{10}  r N^{2/3}\big\}
 \leq e^{-Cr^3}.
 \end{align*} 
\end{proof}

\begin{proof}[Proof of Theorem \ref{fluc}]
If the semi-infinite geodesic $\mathbf{b}^{\rho, (0,0)}$ enters the interior of the square $\lzb v_N-(\delta N^{2/3}, \delta N^{2/3}), v_N\rzb $ as shown in Figure \ref{flucpic}, 
we obtain a $\rho$-geodesic from Proposition \ref{sec3prop1} whose exit time satisfies $\abs{\exittime^{NE, 0\,\to\,v_N}}\le \delta N^{2/3}$. Applying the exit time estimate Theorem  \ref{t:small-ub} finishes the proof.
\end{proof}

\begin{figure}[t]
\captionsetup{width=0.8\textwidth}
\begin{center}
 
\begin{tikzpicture}[>=latex, scale=0.8]

\draw[color = lightgray,line width = 0.3mm, ->] (0,0)--(6,0);
\draw[color = lightgray, line width = 0.3mm, ->] (0,0)--(0,5);

\draw[color=gray, dotted, line width = 1mm, ->] (3,3.36) --(3,5) -- (6,5) -- (6,6);

\draw[color=darkgray, dotted, line width = 1.2mm] (0,0)--(0,0.5) -- (1,0.5) --(1,2) -- (3,2) -- (3,3.6) -- (3.6,3.6);

\draw[color=darkgray, line width = 0.4mm, ->] (3.6,3.6) -- (1,3.6);
\draw[color=darkgray, line width = 0.4mm, ->] (3.6,3.6) -- (3.6,1);

\draw[line width = 0.3mm] (3.5,3.5) -- (2.5,3.5)--(2.5,2.5) --(3.5,2.5) --(3.5,3.5);


\fill[color=white] (3.6,3.6)circle(1.4mm); 
\draw[ fill=white](3.6,3.6)circle(1mm);
\node at (4.2,3.6) {\small $v_N$};

\fill[color=white] (0,0)circle(1.7mm); 
\draw[ fill=lightgray](0,0)circle(1mm);
\node at (-0.3,-0.4) {\small $(0,0)$};

\node at (5,5.5) {\small $\mathbf{b}^{\rho, (0,0)}$};

\end{tikzpicture}

\end{center}
\caption{\small The square in the picture is $\lzb v_N-(\delta N^{2/3}, \delta N^{2/3}), v_N\rzb $. We obtain a $\rho$-geodesic with north and east boundaries from the semi-infinite geodesic in gray.}
\label{flucpic}
\end{figure}

\appendix
\section{Appendix}
Below  is the random walk estimate for the proof of Theorem \ref{t:small-ub}.  It is proved as Lemma C.1 in Appendix C of \cite{balzs2019nonexistence}. 
\begin{lemma}\label{rw}
Let $\alpha > \beta > 0$. Let  $S_n = \sum_{k=1}^n Z_k$ be a random walk with step distribution $Z_k \sim {\rm Exp}(\alpha) - {\rm Exp}(\beta)$ {\rm(}difference of independent exponentials{\rm)}. Then there is an absolute constant $C$ independent of all the parameters such that for $n\in \Z_{>0}$,  
\beq \label{rw34.9}\mathbb{P} (S_1<0, S_2 <0, \cdots , S_n < 0) \leq \frac{C}{\sqrt n} \left(1- \frac{(\alpha-\beta)^2}{(\alpha+\beta)^2} \right)^{\!n} + \frac{\alpha - \beta}{\alpha}.\eeq

\end{lemma}

Next the moment bound on the Radon-Nikodym   for  the proof of Theorem \ref{t:large-ub}.

\begin{lemma} \label{l:radnik} Let  $a>0$, $b\in\R$,  and $N\in\Z_{>0}$.  For $\rho>0$,  let $Q^\rho$ be the  probability distribution on the product space $\Omega=\R^{\fl{aN^{1/3}}}$  under which the coordinates $X_i(\w)=\w_i$  are i.i.d.\ {\rm Exp}$(\rho)$ random variables.   Assume that 
\be\label{appNass}   N\ge   \abs{b}^3 \rho^{-3} (1-\eta)^{-3}  \ee
for some $\eta\in(0,1)$.     Let $f$ denote the Radon-Nikodym derivative 
\[  f(\w)=\frac{d Q^{\rho+bN^{-1/3}}}{dQ^\rho}(\w). \]
Then 
\[  E^{Q^\rho}[ f^2]   
\le \exp\biggl\{  \frac{ab^2}{\rho^2}  +  \frac{10a{\abs b}^3}{3\rho^3\eta N^{1/3}} \biggr\}. 
\]

\end{lemma}

\begin{proof}  Let $\lambda=\rho+bN^{-1/3}$.  Assumption \eqref{appNass} implies that $\abs{\lambda-\rho}\le (1-\eta)\rho$ so in particular the distribution Exp$(\lambda)$ is well-defined.  Note  the inequality 
\be\label{log4}
\biggl\lvert \log (1+x) -x+\frac{x^2}2\,\biggr\rvert \le \sum_{k=3}^\infty \frac{{\abs x}^k}k
\le \frac{{\abs x}^3}{3\eta} 
\ee
valid for $\eta\in(0,1)$ and $\abs x\le 1-\eta$.  Apply it below to $x=b\rho^{-1} N^{-1/3}$ and $x=2b\rho^{-1} N^{-1/3}$.
\begin{align*}
    E^{Q^\rho} [f^2] & = \int_\Omega \biggl(\;\prod_{i=1}^{\floor{aN^{2/3}}} \frac{\lambda e^{-\lambda \omega_i}}{\rho e^{-\rho \omega_i}}\biggr)^2\, Q(d\w) 
    =\left(\frac{\lambda^2}{\rho^2}\int_0^\infty e^{-2(\lambda - \rho)x} \rho e^{-\rho x} dx\right)^{\floor{aN^{2/3}}}\\
    & =
    \left( \frac{\lambda^2}{\rho (2\lambda -\rho)}\right)^{\floor{aN^{2/3}}}
    = \exp\big\{{\floor{aN^{2/3}}} \big[2\log\lambda - \log\rho - \log (2\lambda -\rho) \big] \big\}  \\[5pt] 
    &=\exp\big\{{\floor{aN^{2/3}}} \big[2\log(1+{b\rho^{-1} N^{-1/3}} ) -   \log (1+{2b\rho^{-1} N^{-1/3}} )  \big] \big\}\\[4pt]
    &\le \exp\biggl\{  \frac{ab^2}{\rho^2}  +  \frac{10a{\abs b}^3}{3\rho^3 N^{1/3}} \biggr\}.  
\end{align*}
\end{proof}

\bibliographystyle{plain}
\bibliography{coalbib}

\end{document}